\documentclass[11pt]{article}
\usepackage{format}

\title{Hitting-time mixing for the star transposition shuffle}

\author{
Vanshika Jain\thanks{Department of Mathematics, Princeton University -- Princeton, USA. Email: \href{mailto:vanshika@princeton.edu} {\nolinkurl{vanshika@princeton.edu}}. Research supported in part by an NSF Graduate Research Fellowship DGE-2039656.} \and
Evita Nestoridi \thanks{Stony Brook University, NY 11794. Email: \texttt{evrydiki.nestoridi@stonybrook.edu}. Supported by the Simons Foundation Travel Support for Mathematicians -- MPSTSM00007955 and the NSF grant DMS-2450510.}}
 \date{}

\hypersetup{
    pdftitle={Hitting-time mixing for the star transposition shuffle},
    pdfauthor={Vanshika Jain and Evita Nestoridi}
}

\begin{document}

\maketitle
\begin{abstract}
We prove a hitting-time analogue of cutoff for the star transposition shuffle on
\(\mathfrak S_n\). Let \(\tau\) be the first time at which every non-top card has
been selected. We show that the shuffle is asymptotically mixed at time
\(\tau\): more precisely,
\[
d_{\mathrm{TV}}(\mathcal L(Y_{\tau}),U_{\mathfrak S_n})
\le \exp\left(-(\log n)^{1/2+o(1)}\right).
\]
Our proof compares the star transposition shuffle with the random transposition
shuffle, using simultaneous diagonalization of the two transition kernels, and
then adapts the hitting-time strategy of Jain and Sawhney \cite{jain2024hitting}, introducing a technique that can be applied to card shuffles that are not necessarily conjugacy invariant.
\end{abstract}


\section{Introduction}
Card shuffling has long been a central testing ground for the study of finite
Markov chains. These models provide concrete settings in which to ask
questions about convergence to stationarity, including cutoff and limit
profiles. Classical examples include the riffle shuffle
\cite{bayer1992trailing}, the random transposition shuffle
\cite{diaconis1981generating,teyssier2020limit}, the top-to-random and
random-to-top shuffles \cite{diaconis1992analysis}, the random-to-random
shuffle \cite{bernstein2019cutoff}, random walks generated by \(k\)-cycles
\cite{hough2016random}, and more general conjugacy-invariant walks on
\(\mathfrak S_n\) \cite{berestyckiSengul2019}. More recently, the study of
limit profiles has led to general techniques for reversible Markov chains
\cite{francisNestoridi2026separation, comparison-nestoridi,nestoridiOleskerTaylor2022,teyssier2020limit, Teyssier2026CutoffProfiles}.
The present paper studies a hitting-time refinement of this theory, in which a
natural stopping time captures the statistic governing convergence.

Jain and Sawhney \cite{jain2024hitting} proved such a hitting-time statement
for the random transposition shuffle: if \(\tau\) is the first time every card
has been touched, then the walk is asymptotically mixed at time \(\tau\). Their
argument combines a deterministic-time approximation near the cutoff window
with an approximate-sufficient-statistic result for the untouched set. A key
feature of the random transposition walk is that its increment measure is
conjugacy invariant. Thus Schur's lemma diagonalizes the walk on irreducible
representations, simplifying the analysis, as in the classical
paper of Diaconis and Shahshahani \cite{diaconis1981generating}. 
Related recent work proves hitting-time mixing for random \(k\)-cycles in broad
parameter regimes \cite{shang2026hitting} and determines cutoff profiles for
random transpositions on repeated cards over the full parameter range
\cite{shen2026repeated}.

In this paper, we prove an analogous result for the star transposition shuffle, whose increment
measure is not a class function and hence does not fall within this framework.
We overcome this by comparing the star transposition shuffle to the random transposition shuffle:
the star transposition kernel commutes with the
random transposition kernel and can therefore be simultaneously diagonalized
with it. This comparison lets us transfer deterministic-time estimates to the
star shuffle and adapt the hitting-time strategy of
\cite{jain2024hitting}.

\subsection{Star transposition shuffle}
Jucys--Murphy elements explain particularly clearly why star transpositions
are natural from the viewpoint of representation theory. In the group algebra
\(\mathbb C[\mathfrak S_n]\), the \(k\)-th Jucys--Murphy element is
\[
J_1=0, \qquad 
J_k=\sum_{i<k} (i\,k), \qquad 2\leq k\leq n,
\]
and is therefore built directly out of the star transpositions adjacent to
$k$.  The remarkable fact is that these elements commute and form a set of
simultaneously diagonalizable observables whose eigenvalues encode the contents
of boxes in standard Young tableaux; in this way they give a concrete
``Cartan-like'' replacement for the symmetric group, which has no large
commutative subgroup playing the role of a Cartan subgroup
\cite{Jucys1974,Murphy1981,OkounkovVershik1996}.  This point of view is central
to the Okounkov--Vershik approach to the representation theory of the symmetric
groups, where the branching rule, Young tableaux, and Young's seminormal form
are recovered from the spectra of the Jucys--Murphy elements themselves
\cite{OkounkovVershik1996}.  Moreover, the significance of these elements is
not confined to type $A$: Ram's work shows that analogues of Jucys--Murphy
elements exist for Weyl groups and Iwahori--Hecke algebras, and that their
combinatorics allows one to compute irreducible representations explicitly
\cite{Ram1997,Ram2003}.  Thus, random walks generated by star transpositions
are not merely examples of non-conjugacy-invariant walks; they are generated by
some of the most fundamental algebraic building blocks underlying the
representation theory of symmetric groups and their Coxeter--Hecke
generalizations.

The star transposition shuffle is the Markov chain on permutations obtained by repeatedly swapping a uniformly chosen card with the top card, which corresponds to the Jucys--Murphy element \(X_n\) after appropriate relabeling. Our main question is: starting from the identity, after how many steps is the deck approximately uniform?

Formally, let \(\mathfrak S_n\) denote the symmetric group on \([n] = \{1, 2, \ldots, n\}\). Define the Markov chain \((Y_t)_{t\ge 0}\) on \(\mathfrak S_n\) by \(Y_0=\mathrm{id}\) and, for \(t\ge 1\),
\[
    Y_t=Y_{t-1}(1,I_{t}),
\]
where, for each \(t\), \(I_{t}\) is sampled uniformly from \([n]\) and independently over \(t\). We view a permutation \(\sigma\) as a map from positions to card labels, so \(\sigma(i)\) is the card occupying position \(i\). In other words, at each step, the top card is transposed with a uniformly chosen card, possibly itself. The increments of this Markov chain are distributed according to the probability measure \(\rho^\star\) on \(\mathfrak{S}_n\) given by
\[\rho^\star(\sigma) = \begin{cases}
    1/n &\text{for } \sigma = (1\,i),\, i \in [n] \\
    0 &\text{otherwise}. 
\end{cases}\]
The associated transition kernel is \(Q(x, y) = \rho^\star(x^{-1} y)\).

We measure convergence to stationarity using total variation distance. For probability measures \(\mu\) and \(\nu\) on \(\mathfrak S_n\), define
\[d_{\mathrm{TV}}(\mu, \nu):=\frac12\sum_{\sigma\in \mathfrak S_n}|\mu(\sigma)-\nu(\sigma)|.\]
Let \(\mathcal{L}(Y_t)\) denote the law of \(Y_t\), and let \(U_{\mathfrak{S}_n}\) denote the uniform measure on \(\mathfrak S_n\). Our aim is to understand times \(t\) for which \(d_{\mathrm{TV}}(\mathcal{L}(Y_t), U_{\mathfrak{S}_n})\) is small.

At deterministic times, one can ask not only when the walk becomes close to uniform, but also how abrupt this convergence is and what its asymptotic shape is during the transition. These questions are captured by the notions of cutoff and limit profile, and for the star-transposition shuffle the deterministic-time picture is well understood. Classical work of Flatto, Odlyzko, and Wales identified the spectrum of the walk \cite{flatto1985random}, and Diaconis used this spectral information to show that the star-transposition shuffle exhibits cutoff at time \(n\log n\) with a window of order \(n\) \cite{diaconis1988applications}. The limit profile for star transpositions has also been determined: for each fixed \(c\in\mathbb R\), at time \(t=n(\log n+c)\), the total variation distance converges to \(d_{\mathrm{TV}}(\mathrm{Pois}(1+e^{-c}),\mathrm{Pois}(1))\) \cite{comparison-nestoridi}.

\subsection{Random transposition shuffle}
The same deterministic-time picture holds for the classical random
transposition shuffle. This chain on \(\mathfrak S_n\) is defined by
\(X_0=\mathrm{id}\) and, for \(t\ge 1\),
\[
    X_t=X_{t-1}(I_{t},J_{t}),
\]
where \(I_t\) and \(J_t\) are sampled independently and uniformly from \([n]\) and 
independently over \(t\). The increments of this chain are distributed according to 
the probability measure \(\rho\) on \(\mathfrak S_n\) given by 
\[\rho(\sigma) = \begin{cases}
    1/n &\text{if } \sigma = \text{id}\\
    2/n^2 & \text{if } \sigma = (ij),\, 1 \le i < j \le n\\ 
    0 &\text{otherwise}. 
\end{cases}\] The associated transition kernel is \(P(x, y) = \rho(x^{-1} y)\). Diaconis and Shahshahani proved that the random transposition shuffle exhibits cutoff at time \(\frac12 n\log n\) \cite{diaconis1981generating}. Teyssier \cite{teyssier2020limit} determined the corresponding limit profile, showing that at times \(t=\frac12 n\log n+cn\),
\[
d_{\mathrm{TV}}(\mathcal L(X_t),U_{\mathfrak{S}_n})\to d_{\mathrm{TV}}(\mathrm{Pois}(1+e^{-2c}),\mathrm{Pois}(1)).
\]
Thus Teyssier's theorem sharpens the Diaconis--Shahshahani cutoff result in the same way that Nestoridi's theorem sharpens the cutoff theorem for star transpositions. Jain and Sawhney \cite{jain2024hitting} subsequently established a hitting-time analogue for random transpositions, showing that mixing occurs at a natural random stopping time, the first time all cards have been touched, rather than only at deterministic times. Our result is the corresponding hitting-time statement for the star-transposition shuffle.

\subsection{Main results}
A natural stopping time for the star transposition shuffle is the first time every non-top card has been selected, namely
\[
\tau:=\inf\{t\ge 0:\{2,\dots,n\}\subseteq \{I_1,\dots,I_t\}\}.
\]
This is the earliest time at which one can reasonably expect the walk to be close to stationarity. Indeed, if some \(i\ge 2\) has not appeared among \(I_1,\dots,I_t\), then \(Y_t(i)=i\); before time \(\tau\), the chain is supported on permutations with at least one fixed point. Since a uniformly random permutation has asymptotically \(\mathrm{Pois}(1)\) many fixed points, it has no fixed points with probability \(e^{-1}+o(1)\). Thus the uniform measure assigns positive mass to configurations that the walk cannot reach before \(\tau\), giving a genuine obstruction to mixing. Our main theorem shows that this lower bound is sharp: the walk is already asymptotically uniform when the last untouched card is first selected.

\begin{thm}\label{thm:hitting-time-mixing-star}
    Let \(U_{\mathfrak{S}_n}\) be the uniform distribution on \(\mathfrak{S}_n\). Then 
    \[d_{\mathrm{TV}}(\mathcal L(Y_{\tau}), U_{\mathfrak{S}_n}) \leq \exp\left(-(\log n)^{1/2 + o(1)}\right).\]
\end{thm}

To prove this result, the random-transposition Fourier argument (see Jain--Sawhney \cite{jain2024hitting}) does not directly extend to star transpositions because the star-transposition measure is not a class function. In the random-transposition setting, conjugacy invariance reduces Fourier coefficients to scalars on irreducible representations via Schur's lemma. For the star shuffle, we instead exploit that the star and random transposition kernels commute and are simultaneously diagonalizable, which lets us compare the two chains at matched deterministic times. This strategy is inspired by the comparison argument of \cite{comparison-nestoridi}. A technical complication is that the comparison is sensitive to parity. The relevant eigenvalue terms may change sign, and the comparison works cleanly only when the two times being matched have the same parity. This leads to an even-time comparison theorem.

\begin{thm}\label{thm:even-time-comparison-thm}
    Let $t=\lfloor \tfrac{n\log n}{2}\rfloor+t'$ be even, with $|t'|\le n\bigl(\tfrac{\log\log n}{4}-\log\log\log n\bigr)$, and set $s:=2t$. Then
    \[d_{\mathrm{TV}}(X_t,Y_s)\le n^{-1/2+o(1)}.\]
\end{thm}

We then use stability estimates to move from this parity-restricted statement to general nearby times. Combining the restricted parity statement, the stability results, and the discrete time approximation of \(X_t\) by \(\nu_t\) \cite{jain2024hitting}, we obtain a discrete time approximation for the star transposition shuffle (Corollary~\ref{cor:final-tv-distance-star-walk}). To state it, we introduce some notation (notation is drawn from \cite{jain2024hitting}).

Let 
\begin{equation}\label{eq:def-t-n}
    t_n = \lfloor (n \log n)/2 \rfloor, \quad t = t_n + n c_n(t)
\end{equation}
where \(nc_n(t) \in \mathbb{Z}\). Define
\begin{equation}\label{eq:def-gamma-n-t}
\gamma_{n, t} = e^{-2(t-t_n)/n} = e^{-2 c_n(t)}.
\end{equation}
For star-shuffle times, we use the natural normalization
\begin{equation}\label{eq:def-star-window}
    s_n:=\lfloor n\log n\rfloor,\qquad
    c_n^\star(s):=\frac{s-s_n}{n}.
\end{equation}

\begin{defn}\label{def:nu-t}
    Let \(\nu_t\) be the probability measure on \(\mathfrak{S}_n\) generated as follows. First, sample \(M_t \in \{0, 1, \ldots, n\}\) from the truncated Poisson law \[\bbP(M_t = x) = \frac{\bbP(\mathrm{Pois}(\gamma_{n, t}) = x)}{\bbP(\mathrm{Pois}(\gamma_{n, t}) \leq n)}.\] Then, choose a uniformly random subset \(S_t \subset [n]\) of size \(M_t\). Finally, sample a uniformly random permutation on \([n]\setminus S_t\) and extend it to an element of \(\mathfrak{S}_n\) by fixing every element of \(S_t\).
\end{defn}

\begin{cor}\label{cor:final-tv-distance-star-walk}
    Let \(s\in\mathbb Z\) satisfy
    \[
        |c_n^\star(s)|\le \frac12\log\log n-2\log\log\log n,
    \]
    and set \(t:=\lceil s/2\rceil\). Then
    \[d_{\mathrm{TV}}(Y_s,\nu_t)\le n^{-1/2+o(1)}.\]
\end{cor}

This corollary should be regarded as a quantitative strengthening of the cutoff-profile theorem for star transpositions from \cite{comparison-nestoridi}. The result compares the full law of the walk to an explicit measure \(\nu_t\). We not only identify the asymptotic distance to stationarity, but we also give an explicit distribution which approximates the walk itself. As in Jain--Sawhney's strengthening of Teyssier's theorem for random transpositions, the cutoff profile is then recovered from this finer approximation by the triangle inequality.

Finally, to deduce the hitting-time theorem, we follow the general blueprint of Jain and Sawhney. The deterministic-time approximation by \(\nu_t\) is bootstrapped into a statement that the untouched set is an approximate sufficient statistic: conditioned on the untouched cards, the deck is close to uniform on permutations fixing those cards. This conditional description is then merged with a strong-uniform-time argument to prove the hitting-time mixing result (Theorem~\ref{thm:hitting-time-mixing-star}).

\section{Deterministic-time approximation for star transposition walk} \label{sec:deterministic-time-approx-for-star}
This section proves a deterministic time approximation for the star transposition walk. Our goal is to compare the law of the star walk \(Y_s\) at the relevant time scale to the approximation \(\nu_t\) arising from the random transposition setting. The proof has three steps. First, using a comparison lemma together with simultaneous diagonalization of the random transposition and star transposition kernels, we prove an even-time comparison between \(X_t\) and \(Y_{2t}\). Second, we remove the parity restriction by proving stability estimates for both the model law \(\nu_t\) and the star walk under small changes in time. Finally, combining these ingredients with Jain-Sawhney's deterministic-time approximation for random transpositions yields the desired approximation for $Y_s$.

\subsection{Spectral set-up} \label{subsection:spectral-set-up}

\subsubsection{Comparison lemma}
Let \((X_r)_{r\ge0}\) and \((Y_r)_{r\ge0}\) be two irreducible, reversible Markov chains on the same finite state space \(\mathcal S\), both with stationary distribution \(\pi\). We compare their distributions at times \(t\) and \(s\) by viewing their one-step transition kernels as matrices, $P$ and $Q$ respectively.
\begin{lem}[{\cite[Lemma 1.4]{comparison-nestoridi}}] \label{lem:bounding-TV-eigen} Let \(P\) and \(Q\) be the transition matrices of two reversible Markov chains \(X_t\) and \(Y_s\) on a finite state space \(\cS\) that share the same eigenbasis and stationary measure \(\pi\). Let \(\varphi_i: \cS \to \bbC\) be a common orthonormal eigenbasis of \(\ell^2(\pi)\), with inner product
$\langle f,g\rangle_\pi:=\sum_{x\in\cS}f(x)\overline{g(x)}\pi(x),$
satisfying \[P \varphi_i = \alpha_i \varphi_i, \quad \text{and} \quad Q \varphi_i = \beta_i \varphi_i.\]
For every \(x \in \cS\), and \(t, s \geq 0\), we have \[4\|P^t(x,\cdot)-Q^{s}(x,\cdot)\|_{TV}^2
\;\le\;
\sum_{i\ge2}^{|\cS|} |\varphi_i(x)|^2\bigl(\alpha_i^{t}-\beta_i^{s}\bigr)^2.
\]
If $P$ and $Q$ are random walks on a group $S$, then
\[4\|P^t(x,\cdot)-Q^{s}(x,\cdot)\|_{TV}^2
\;\le\;
\sum_{i\ge2}^{|\cS|} \bigl(\alpha_i^{t}-\beta_i^{s}\bigr)^2. 
\] 
\end{lem}

\subsubsection{Representation-theoretic preliminaries}
For the two walks under consideration, \(P\) is the random-transposition kernel and \(Q\) is the star-transposition kernel. Both have stationary distribution \(U_{\mathfrak S_n}\), and both are reversible with respect to it. Hence the associated operators are self-adjoint on \(\ell^2(U_{\mathfrak S_n})\). Thus, to apply Lemma~\ref{lem:bounding-TV-eigen}, it suffices to verify that \(P\) and \(Q\) commute; in that case, they admit a common orthonormal eigenbasis.

\begin{lem} [{\cite[Lemma 4.8]{comparison-nestoridi}}] \label{lem:commuting-matrices}
    Let \(G\) be a finite group, and \(\mu, \nu\) probability measures on \(G\) with \(\mu\) constant on conjugacy classes. Define two transition kernels on \(G\) by \(P(x, y) = \mu(x^{-1}y)\) and \(Q(x, y) = \nu(x^{-1}y)\). If \(\nu\) is symmetric, then \(P\) and \(Q\) commute. 
\end{lem}

Since $\rho$ is class-invariant and $\rho^\star$ is symmetric, Lemma~\ref{lem:commuting-matrices} implies that $P$ and $Q$ commute. It follows that $P$ and $Q$ share a common orthonormal eigenbasis. 
In our setting, the joint eigenspaces admit an explicit representation-theoretic description; we record this in Lemma~\ref{lem:common-eigenbasis}. 
To state it, we introduce the necessary notation.

A partition of $n$ is a sequence $\lambda=(\lambda_1,\ldots,\lambda_k)$, written $\lambda\vdash n$, with $\lambda_1\ge \lambda_2\ge \cdots \ge \lambda_k>0$ and $\sum_{i=1}^k \lambda_i=n$. We adopt the convention that $\lambda_i=0$ for $i>k$.
Every partition $\lambda$ determines a Young diagram with $\lambda_i$ boxes in row $i$. 
Let $\lambda'$ be the transpose partition, whose $i$th part $\lambda'_i$ equals the number of boxes in the $i$th column of the Young diagram of $\lambda$.
A standard Young tableau of shape $\lambda$ is a filling of the Young diagram of $\lambda$ with $[n]$, where each element is used exactly once and entries increase strictly along rows (left to right) and columns (top to bottom).
Let $d_\lambda$ denote the number of standard Young tableaux of shape $\lambda$.

The eigenvalues of \(P\) are indexed by partitions \(\lambda\vdash n\), and their multiplicities are expressed in terms of the dimensions \(d_\lambda\). We now record the eigenvalue formulas for \(P\) and gather the accompanying bounds and multiplicity estimates that will be used later. Some of the lemmas below include multiple citations: in these cases, the notation and formulation are often taken from the first one or two sources, while the proof of the original statement appears in the final cited reference.

\begin{lem}[{\cite[Lemma 4.5]{comparison-nestoridi}}, {\cite[Lemma 3.2]{jain2024hitting}}, {\cite[Lemma 7]{diaconis1981generating}}] \label{lem:eigenvalue-transp-shuffle}
    Let \(\lambda \vdash n\) with \(\lambda = (\lambda_1, \ldots, \lambda_k)\). 
    The eigenvalues of \(P\), the transition kernel associated to the random transposition walk, are given by \[\alpha_\lambda = \frac{1}{n} + \frac{n-1}{n} \cdot \binom{n}{2}^{-1} \sum_{i\ge 1} \left[ \binom{\lambda_i}{2} - \binom{\lambda_i'}{2} \right].\] 
    If \(\lambda_1 \geq n - (\log n)^2 \), then 
    \[\alpha_\lambda = 1 - \frac{2(n - \lambda_1)}{n} + O(n^{-2+o(1)}).\]
    The spectral block indexed by \(\lambda\) has dimension \(d_\lambda^2\); equivalently, it contributes \(d_\lambda^2\) copies of \(\alpha_\lambda\) to the spectral multiset.
\end{lem}

Let $\lambda\vdash n$. If the rightmost box in row $i$ of the Young diagram of $\lambda$ is a corner, let $\lambda^{(i)}\vdash (n-1)$ denote the partition obtained by removing the corner box in row $i$. The eigenvalues of $Q$ are naturally indexed by such corner removals $\lambda^{(i)}$.

\begin{lem}[{\cite[Lemmas 5.1 and 5.4]{comparison-nestoridi}}, {\cite[Theorems 3.6 and 3.7]{flatto1985random}}]\label{lem:eigenvalues-star-transp}
    Let $\lambda\vdash n$, and let $i$ be a row whose rightmost box is a corner. Then $Q$, the star-transposition kernel on $\mathfrak S_n$, has an eigenvalue
    \[ \beta_{\lambda^{(i)}}=\frac{1}{n}\,(\lambda_i-i+1) \]
    and the spectral block indexed by $(\lambda,\lambda^{(i)})$ has dimension $d_\lambda\,d_{\lambda^{(i)}}$. Moreover, if $i>1$, then
    \[-\frac{n-\lambda_1}{n}\le \beta_{\lambda^{(i)}}\le \frac{\lambda_1}{n}.\]
\end{lem}

The next lemma gives an explicit parametrization of the joint eigenspaces (and hence of a common eigenbasis) for \(P\) and \(Q\). This allows us to compare \(Y_s\) and \(X_t\). 

\begin{lem}[{\cite[Lemma 5.2]{comparison-nestoridi}}]\label{lem:common-eigenbasis}
    Let \(\lambda \vdash n\) and let \(I(\lambda) = \{i:\lambda_i > \lambda_{i+1}\}\) be the set of rows whose rightmost box is a corner. Let \(P\) be the transition kernel for the random transposition walk, and let \(Q\) be the transition kernel for the star-transposition walk. For each \(i \in I(\lambda)\), there is a joint eigenspace of \(P\) and \(Q\) indexed by \((\lambda, \lambda^{(i)})\), on which \(P\) acts by \(\alpha_\lambda\) and \(Q\) acts by \(\beta_{\lambda^{(i)}}\). The joint eigenspace has dimension \(d_\lambda d_{\lambda^{(i)}}\).
\end{lem}

We now record the estimates for $d_\lambda$ and $d_{\lambda^{(i)}}$.

\begin{lem}[{\cite[Lemma 3.3]{jain2024hitting}}, {\cite[Corollary 2]{diaconis1981generating}}]\label{lem:bounds-d-lambda}
    Let \(\lambda \vdash n\). Then \(d_\lambda \leq \binom{n}{\lambda_1} \sqrt{(n-\lambda_1)!}\).
\end{lem}

To compare $d_\lambda$ and $d_{\lambda^{(i)}}$, we invoke the branching rule for irreducible representations of $\mathfrak{S}_n$.

\begin{lem}[{\cite[Theorem 3.6]{flatto1985random}}] \label{lem:branching-rule}
    Let \(\lambda \vdash n\). Then
    $d_\lambda = \sum_{i\in I(\lambda)} d_{\lambda^{(i)}}.$
    Moreover, for every \(i\in I(\lambda)\),
    $d_{\lambda^{(i)}} = d_{(\lambda')^{(\lambda_i)}}.$
\end{lem}

The following lemma provides an upper bound comparing \(d_{\lambda^{(i)}}\) to \(d_\lambda\) when the removed corner lies outside the first row. It improves the bound in \cite[Lemma 5.3]{comparison-nestoridi}, replacing the factor $4^x/n$, where $x=n-\lambda_1$, by $i\lambda_i/n$.

\begin{lem} \label{lem:bounding-di-vals}
    Let \(\lambda \vdash n\), and let \(i\in I(\lambda)\) satisfy \(i > 1\). Then \[d_{\lambda^{(i)}} \leq \frac{i \lambda_i}{n} d_\lambda.\]
\end{lem}

\begin{proof}
    The hook \(h_{i, j}\) of the box in row \(i\), column \(j\) of a Young diagram for \(\lambda\) is the number of boxes directly to its right in the same row, plus the number of boxes directly below it in the same column, plus one for the box itself. The hook-length formula tells us {\cite[Theorem 3.10.2]{sagan2013symmetric}} \[d_\lambda = \frac{n!}{\prod_{(i, j) \in \lambda} h_{i, j}}.\]

    For \(\lambda^{(i)}\) to be a valid partition, the rightmost box in row \(i\) must be a corner. Since the last box in row $i$ has no boxes to its right or below it, its hook-length is $h_{i,\lambda_i}=1$. Removing that box reduces the hook-length by one for each box in row $i$ to its left and for each box in column $\lambda_i$ above it. Hence one finds \[\frac{d_{\lambda^{(i)}}}{d_\lambda} = \frac{1}{n} \prod_{k = 1}^{\lambda_i-1} \frac{h_{i, k}}{h_{i, k}-1} \prod_{\ell = 1}^{i-1} \frac{h_{\ell, \lambda_i}}{h_{\ell, \lambda_i}-1}.\] Note that at position $(i,k)$ there are $\lambda_i - k$ boxes to its right, so \(h_{i,k} \ge \lambda_i - k + 1\) and the factor \[\frac{h_{i, k}}{h_{i, k} -1} \leq \frac{\lambda_i - k + 1}{\lambda_i - k}.\] Similarly, at position \((\ell, \lambda_i)\), there are $i - \ell$ boxes below it, so \(h_{\ell, \lambda_i} \ge i - \ell + 1\) and the factor \[\frac{h_{\ell, \lambda_i}}{h_{\ell, \lambda_i} -1} \leq \frac{i - \ell + 1}{i - \ell}.\] Combining these observations we get \[\frac{d_{\lambda^{(i)}}}{d_\lambda} \leq \frac{1}{n} \frac{\lambda_i!}{(\lambda_i - 1)!} \frac{i!}{(i-1)!} = \frac{i\lambda_i}{n}.\]
\end{proof}

\subsection{Even-time comparison theorem} \label{subsection:even-time-comparison-thm}
In this subsection, our main goal is to prove that 
$d_{\mathrm{TV}}(X_t,Y_s)\le n^{-1/2+o(1)} $
for even \(t\) and \(s = 2t\) where \(X_t\) is the random transposition walk and \(Y_s\) is the star transposition walk. By Lemma~\ref{lem:bounding-TV-eigen}, it suffices to bound a spectral sum indexed by pairs \((\lambda,\lambda^{(i)})\), arising from the eigenvalues of the random- and star-transposition kernels. We decompose this sum into three regimes according to the size of the first row and first column of \(\lambda\). The regime \(\mathcal S\), where both the first row and column are small, makes a negligible contribution to the sum. The regime \(\mathcal L\), corresponding to partitions with large first row, is further divided into the cases \((\lambda,\lambda^{(1)})\) and \((\lambda,\lambda^{(i)})\) for \(i>1\); the latter is the main contribution to the argument. The remaining regime \(\mathcal L'\), corresponding to partitions with large first column, is handled by transposing partitions and reducing to the analysis of \(\mathcal L\).

\subsubsection{Reduction to three regimes}\label{subsection:reduction-to-three-regimes}
The next two results reduce the main spectral sum from Lemma~\ref{lem:bounding-TV-eigen} to three regimes and, in the \(\mathcal S\) regime, identify a useful partition pairing. We write
\[I(\lambda)=\{\,i:\lambda_i>\lambda_{i+1}\,\}\] 
for the set of rows whose rightmost cell is a corner.

\begin{cor}\label{cor:spectral-comparison-Sn}
    Let $P$ be the random transposition kernel and $Q$ the star-transposition kernel on $\mathfrak S_n$.
    Then for any $\sigma\in\mathfrak S_n$ and any $t,s\ge 0$,
    \[
    4\|P^t(\sigma,\cdot)-Q^{s}(\sigma,\cdot)\|_{\mathrm{TV}}^2
    \;\le\;
    \sum_{\lambda\vdash n} d_\lambda \sum_{i\in I(\lambda)}
    d_{\lambda^{(i)}}\,\bigl(\alpha_\lambda^{\,t}-\beta_{\lambda^{(i)}}^{\,s}\bigr)^2,
    \]
    where $\alpha_\lambda$ are the eigenvalues of $P$, $\beta_{\lambda^{(i)}}$ are the eigenvalues of $Q$,
    and $d_\lambda d_{\lambda^{(i)}}$ is the dimension of the joint eigenspace indexed by $(\lambda,\lambda^{(i)})$.
\end{cor}

\begin{proof}
    Since the random-transposition kernel \(P\) and the star-transposition kernel \(Q\) commute and admit a common eigenbasis (Lemmas~\ref{lem:commuting-matrices} and~\ref{lem:common-eigenbasis}), Lemma~\ref{lem:bounding-TV-eigen} applies. Moreover, left multiplication by \(\mathfrak S_n\) acts transitively on the state space and preserves both right-convolution kernels. Thus the two chains are transitive under a common action, and for any state \(x \in \mathfrak{S}_n\) and any times \(t,s \geq 0\),
    \[4\|P^t(x,\cdot)-Q^{s}(x,\cdot)\|_{TV}^2
    \;\le\;
    \sum_{i \geq 2}^{n!}\bigl(\alpha_i^{t}-\beta_i^{s}\bigr)^2. \] 

    By Lemma~\ref{lem:common-eigenbasis}, the joint spectral blocks of \(P\) and \(Q\) can be enumerated by the pairs \((\lambda,\lambda^{(i)})\) with \(\lambda\vdash n\) and \(i\in I(\lambda)\), each having dimension \(d_\lambda d_{\lambda^{(i)}}\). The trivial block contributes zero to the difference, so we may rewrite
    \begin{equation}\label{eq:double-sum-corner-partitions}
        \sum_{j=2}^{n!}\bigl(\alpha_j^{\,t}-\beta_j^{\,s}\bigr)^2
    \;=\;
    \sum_{\lambda\vdash n} 
       \;d_\lambda\;\sum_{i\in I(\lambda)} 
          d_{\lambda^{(i)}}\,
          \bigl(\alpha_\lambda^{\,t}-\beta_{\lambda^{(i)}}^{\,s}\bigr)^2.
    \end{equation}
\end{proof}

Using the bijection between corners of $\lambda$ and corners of its transpose $\lambda'$, we may pair each index $i\in I(\lambda)$ with the corresponding corner in $\lambda'$. Under this pairing, corners with $\lambda_i<i$ (``negative'' corners) are dominated by their partners with $\lambda_i\ge i$. 

\begin{lem} \label{lem:pairing-betas}
    Let $\lambda\vdash n$ and let $i\in I(\lambda)$ satisfy \(\lambda_i \geq i\). Then \(\beta_{\lambda^{(i)}} \ge|\beta_{(\lambda')^{(\lambda_i)}}|\).
\end{lem}

\begin{proof}
    By Lemma~\ref{lem:eigenvalues-star-transp}, observe that for any row \(i\) whose rightmost cell is a corner, \[\beta_{\lambda^{(i)}} =\frac1n+\frac{n-1}n\,b_{\lambda^{(i)}}, \qquad b_{\lambda^{(i)}} =\frac{\lambda_i-i}{\,n-1\,}, \] and that under diagram transpose \[ b_{(\lambda')^{(\lambda_i)}} =-\,b_{\lambda^{(i)}}. \]
    Therefore, if \(b_{\lambda^{(i)}} \ge 0\) (equivalently \(\lambda_i \geq i\)), we get that \( \beta_{\lambda^{(i)}} =|\beta_{\lambda^{(i)}}| \;\ge\;|\beta_{(\lambda')^{(\lambda_i)}}|. \)
\end{proof}

Consequently, for the transpose-invariant sums of nondecreasing functions of \(|\beta_{\lambda^{(i)}}|\) used below, it suffices to sum only over the set of positive corners
\begin{equation} \label{eq:J(lambda)}
J(\lambda)\;:=\;\{\,i\in I(\lambda): \lambda_i\ge i\,\}.
\end{equation}
We then multiply the resulting bound by \(2\).
To apply this pairing observation and split our sums based on the size of \(\lambda_1\), we need both \(\lambda\) and \(\lambda'\) to have similar size. Accordingly, we classify every partition into one of three types: those whose largest part is \enquote{large}, those whose transpose's largest part is \enquote{large}, and those whose largest part and whose transpose's largest part are both \enquote{small}. Define
\begin{equation}\label{eq:def-S-L-L'}
\mathcal L :=\bigl\{\lambda\vdash n : \lambda_1 \geq n - (\log n)^2\bigr\},
\quad
\mathcal L' :=\bigl\{\lambda\vdash n : \lambda'_1 \geq n - (\log n)^2\bigr\}, 
\quad 
\mathcal S := \bigl\{\lambda\vdash n : \lambda \not\in \mathcal L \cup \mathcal L'\bigr\}. 
\end{equation}
For all sufficiently large \(n\), the sets \(\mathcal L\) and
\(\mathcal L'\) are disjoint because
\(\lambda_1+\lambda'_1\le n+1\). Thus these three sets partition
\(\{\lambda\vdash n\}\), and we treat the corresponding sums separately.

\subsubsection{Bounding the \texorpdfstring{$\mathcal{S}$}{S} regime}

We bound the sum in the \(\mathcal S\) regime using Lemma~\ref{lem:random-transp-eigenval-bound}, which controls the random-transposition contribution \(\sum d_\lambda^2|\alpha_\lambda|^{2t},\) and Lemma~\ref{lem:bounding-d-di-beta-2s}, which controls the star-transposition contribution
\[\sum d_\lambda\sum_{i\in J(\lambda)} d_{\lambda^{(i)}}\beta_{\lambda^{(i)}}^{2s}.\] 
Combining these estimates in Lemma~\ref{lem:S-regime} gives an \(n^{-2+o(1)}\) contribution from \(\mathcal S\).

\begin{lem}[{\cite[Lemma 3.4]{jain2024hitting}}] \label{lem:random-transp-eigenval-bound}
    Let \(t = \lfloor \frac{n \log n}{2} \rfloor + t'\) be an integer and \(|t'| \leq n(\frac{\log \log n}{4} - \log \log \log n)\). Let \(\alpha_\lambda\) be the eigenvalues of \(P\), the transition kernel associated to the random transposition walk. Then
    \[\sum\limits_{\substack{\lambda \vdash n \\ \lambda_1 < n-(\log n)^2}} d_\lambda^2 |\alpha_\lambda|^{2t} \leq n^{-2 + o(1)}. \]
\end{lem}

It remains to obtain the corresponding estimate for the star-transposition part. The next lemma shows that, after summing over the relevant corner removals indexed by
\[J(\lambda)=\{\,i\in I(\lambda):\lambda_i\ge i\,\},\]
as defined in \eqref{eq:J(lambda)}, the star-transposition contribution from \(\mathcal S\) is superpolynomially small.

\begin{lem} \label{lem:bounding-d-di-beta-2s}
    Let \(s\) be an integer satisfying
    \[
        \left|\frac{s-\lfloor n\log n\rfloor}{n}\right|
        \leq \frac{\log\log n}{2}-2\log\log\log n+\frac{5}{n}.
    \]
    Let \(\beta_{\lambda^{(i)}}\) be the eigenvalues of \(Q\), the transition kernel associated to the star transposition walk. Then
    \[\sum_{\lambda\in\cS}d_\lambda\sum_{i\in J(\lambda)}d_{\lambda^{(i)}}\beta_{\lambda^{(i)}}^{2s} \leq n^{-\omega(1)}. \]
\end{lem}

\begin{proof}
    For \(i\in J(\lambda)\), Lemma~\ref{lem:eigenvalues-star-transp} gives
    \(0<\beta_{\lambda^{(i)}}\le(n-x)/n\), where
    \(x=n-\lambda_1\). By Lemma~\ref{lem:branching-rule},
    \(\sum_{i \in J(\lambda)} d_{\lambda^{(i)}} \leq
    \sum_{i \in I(\lambda)} d_{\lambda^{(i)}} = d_\lambda\).
    Putting these bounds together, we get the following:
    \begin{equation*} 
    2\sum_{\lambda\in\cS}d_\lambda\sum_{i\in J(\lambda)}d_{\lambda^{(i)}}\beta_{\lambda^{(i)}}^{2s} 
    \le 2\sum_{\lambda\in\cS}d_\lambda\bigl(\tfrac{n-x}{n}\bigr)^{2s}\sum_{i\in J(\lambda)}d_{\lambda^{(i)}} 
    \;\le\; 2\sum_{\lambda\in\cS}d_\lambda^2\bigl(\tfrac{n-x}{n}\bigr)^{2s}.
    \end{equation*} 
    By Lemma~\ref{lem:bounds-d-lambda}, \(d_\lambda\le\binom n x\sqrt{x!}\). Let \(p(x)\) count the number of partitions of \(x\). Since there are at most \(p(x)\) partitions with \(n-\lambda_1=x\), we may re-index the sum by \(x\ge(\log n)^2\), giving
    \[2\sum_{\lambda\in\cS}d_\lambda^2\bigl(\tfrac{n-x}{n}\bigr)^{2s} \le 2\sum_{x\ge(\log n)^2}p(x)\binom n x^2x!\bigl(\tfrac{n-x}{n}\bigr)^{2s}.\]  
 
    Taylor approximation tells us \(\log(1-u) \leq -u\) for \(u \in (0, 1)\), so we have
    \[
        \bigl(\tfrac{n-x}{n}\bigr)^{2s}
        \le \exp\left(\frac{-2sx}{n}\right)
        \le e^{O(x/n)}n^{-2x}(\log n)^x(\log\log n)^{-4x}.
    \]
    The partition function has upper bound given by the Hardy--Ramanujan formula: \(p(x)\le e^{\pi\sqrt{2x/3}}\).
    Moreover \(\binom n x^2x! \leq \frac{n^{2x}}{x!}\). Stirling's approximation formula gives us a lower bound on \(x!\): \(x!\ge(x/e)^x\sqrt{2\pi x}\). Combining these bounds, we get   
    \begin{align*}
        2\sum_{x\ge(\log n)^2}p(x)\binom n x^2x!\bigl(\tfrac{n-x}{n}\bigr)^{2s}
        &\le\; e^{O(1)}\sum_{x\ge(\log n)^2} e^{\pi\sqrt{2x/3}} \,\, \frac{n^{2x}}{x!} \,\, n^{-2x}(\log n)^x(\log\log n)^{-4x} \\
        &\le\; e^{O(1)}\sum_{x\ge(\log n)^2}\Bigl(\frac{e^2\log n}{x(\log\log n)^4}\Bigr)^x =\;n^{-\omega(1)}.
    \end{align*}
\end{proof}

The preceding two lemmas give separate bounds for the random-transposition and star-transposition parts of the \(\mathcal S\)-sum; the next lemma combines them.

\begin{lem}[Contribution of the $\mathcal S$ regime]\label{lem:S-regime}
    Under the assumptions of Lemmas~\ref{lem:random-transp-eigenval-bound} and~\ref{lem:bounding-d-di-beta-2s}, we have
    \[\sum_{\lambda\in\mathcal S} d_\lambda \sum_{i\in I(\lambda)} d_{\lambda^{(i)}}
    \bigl(\alpha_\lambda^{\,t}-\beta_{\lambda^{(i)}}^{\,s}\bigr)^2
    \le n^{-2+o(1)}.\]
\end{lem}

\begin{proof}
    Starting from the right-hand side of Corollary~\ref{cor:spectral-comparison-Sn}, we expand the square into two diagonal terms and one cross term. We then apply Lemma~\ref{lem:branching-rule} to rewrite it as a sum over \(\lambda\) alone. This yields a decomposition into three parts:
    \begin{align}
    \sum_{\lambda \in \cS} \;d_\lambda\;\sum_{i\in I(\lambda)} d_{\lambda^{(i)}}\, \bigl(\alpha_\lambda^{\,t}-\beta_{\lambda^{(i)}}^{\,s}\bigr)^2 
    &\leq \sum_{\lambda \in \cS} \;d_\lambda\;\sum_{i\in I(\lambda)} d_{\lambda^{(i)}}\,
    \bigl(|\alpha_\lambda|^{2t} + 2|\alpha_\lambda|^t|\beta_{\lambda^{(i)}}|^s + |\beta_{\lambda^{(i)}}|^{2s}\bigr) \nonumber\\ 
    &\leq \sum_{\lambda \in \cS}d_\lambda\sum_{i\in I(\lambda)} d_{\lambda^{(i)}} |\beta_{\lambda^{(i)}}|^{2s}
    +\sum_{\lambda \in \cS} \;d_\lambda^2 |\alpha_\lambda|^{2t} \nonumber\\
    &\qquad+ 2\sum_{\lambda \in \cS}d_\lambda |\alpha_\lambda|^t\;\sum_{i\in I(\lambda)} d_{\lambda^{(i)}}|\beta_{\lambda^{(i)}}|^s \label{eq:small-lambda-divide-3-parts}.
    \end{align}
    The first sum, after pairing via Lemma~\ref{lem:pairing-betas}, is bounded by \(n^{-\omega(1)}\) using Lemma~\ref{lem:bounding-d-di-beta-2s}:
    \begin{equation} \label{eq:final-d^2-beta^2-bound}
    \sum_{\lambda\in\cS}\;d_\lambda\sum_{i\in I(\lambda)}d_{\lambda^{(i)}}|\beta_{\lambda^{(i)}}|^{2s} \;\leq\; 2\sum_{\lambda\in\cS}d_\lambda\sum_{i\in J(\lambda)}d_{\lambda^{(i)}}\beta_{\lambda^{(i)}}^{2s} \leq n^{-\omega(1)}.
    \end{equation}
    For the second term in Equation~\ref{eq:small-lambda-divide-3-parts}, we use Lemma~\ref{lem:random-transp-eigenval-bound} to get \begin{equation} \label{eq:applying-random-trans-bound}
        \sum_{\lambda \in \cS} d_\lambda^2 |\alpha_\lambda|^{2t} \leq n^{-2+o(1)}.
    \end{equation}

    We now consider the final term in Equation~\ref{eq:small-lambda-divide-3-parts}: \[\sum_{\lambda \in \cS} d_\lambda |\alpha_\lambda|^t \sum_{i \in I(\lambda)} d_{\lambda^{(i)}} |\beta_{\lambda^{(i)}}|^s.\] Applying Cauchy--Schwarz over the pairs \((\lambda,i)\), we have
    \[\left(\sum_{\lambda \in \cS} d_\lambda |\alpha_\lambda|^t \sum_{i \in I(\lambda)} d_{\lambda^{(i)}} |\beta_{\lambda^{(i)}}|^s\right)^2 \leq \left(\sum_{\lambda \in \cS} d_\lambda^2 |\alpha_\lambda|^{2t}\right) \left(\sum_{\lambda \in \cS}d_\lambda\sum_{i \in I(\lambda)} d_{\lambda^{(i)}} |\beta_{\lambda^{(i)}}|^{2s}\right).\]
    By Lemma~\ref{lem:random-transp-eigenval-bound}, the first term on the right hand side is bounded by \(n^{-2+o(1)}\). The second term is bounded by Equation~\ref{eq:final-d^2-beta^2-bound}. Thus, our original sum is at most \begin{equation} \label{eq:cauchy-combining-two}
        \sqrt{n^{-\omega(1)}n^{-2+o(1)}} = n^{-\omega(1)}.
    \end{equation}
    Combining Equations~\ref{eq:final-d^2-beta^2-bound},
    \ref{eq:applying-random-trans-bound}, and
    \ref{eq:cauchy-combining-two}, we obtain the bound \(n^{-2+o(1)}\) in
    Equation~\ref{eq:small-lambda-divide-3-parts}.
\end{proof}

\subsubsection{Bounding the \texorpdfstring{$\mathcal{L}$}{L} regime}
In the \(\mathcal L\) regime, we have \(\lambda_1 \ge n-(\log n)^2\), so \(\lambda\) is close to the one-row partition. Here the relevant eigenvalues may be large (in the sense that they can lie close to \(1\)) but their multiplicities are controlled by the small number \(n-\lambda_1\) of boxes outside the first row. The common spectral decomposition of the random- and star-transposition kernels is indexed by corner removals \(\lambda^{(i)}\), and we split the \(\mathcal L\)-sum according to whether the corner lies in the first row \((i=1)\) or in a lower row \((i>1)\). We record these two estimates as Lemmas~\ref{lem:L-i=1} and~\ref{lem:L-i>1}, and then combine them in Lemma~\ref{lem:L-regime}.

\begin{lem}\label{lem:L-i>1}
    Let $t=\lfloor \tfrac{n\log n}{2}\rfloor+t'$ be an integer with $|t'|\le n(\tfrac{\log\log n}{4}-\log\log\log n)$ and let $s=2t$.
    Then, \[\sum_{\lambda \in \cL} \;d_\lambda\;\sum_{i\in I(\lambda), i >1} d_{\lambda^{(i)}}\, \bigl(\alpha_\lambda^{\,t}-\beta_{\lambda^{(i)}}^{\,s}\bigr)^2 \leq n^{-1+o(1)}.\]
\end{lem}

\begin{proof}
    By Lemma~\ref{lem:eigenvalue-transp-shuffle}, for \(n\) sufficiently large and \(\lambda \in \cL\), \(\alpha_\lambda = 1 - \tfrac{2x}{n} + O(n^{-2+o(1)})\). Therefore, \[\alpha_\lambda^t = \exp\left(\frac{-2tx}{n}\right)(1 + n^{-1+o(1)}).\]

   In order to bound \[\bigl(\alpha_\lambda^{\,t}-\beta_{\lambda^{(i)}}^{\,s}\bigr)^2,\] we first observe from Lemma~\ref{lem:eigenvalues-star-transp} that \(\beta_{\lambda^{(i)}}\) always lies between \(\tfrac{-x}{n}\) and \(\tfrac{n-x}{n}\). Since the map \(b \mapsto (\alpha_\lambda^t - b^s)^2\) is continuous on the closed interval \([\tfrac{-x}{n}, \tfrac{n-x}{n}]\), its maximum must occur either at an interior critical point or at one of the two endpoints. Differentiating shows that the interior critical points are \(b = 0, \pm \sqrt{\alpha_\lambda}\). The values of squared difference at these points are \(\alpha_{\lambda}^{2t} = \exp(-4tx/n)(1 + n^{-1+o(1)})\) and \(0\) respectively.
    At the right endpoint, set \(u=(1-x/n)^2\), so that
    \(((n-x)/n)^s=u^t\). If \(x=0\), then \(\lambda=(n)\) and
    \(\alpha_\lambda=u=1\), so the difference vanishes. For \(x\ge1\)
    and \(n\) sufficiently large, both \(\alpha_\lambda\) and \(u\) are
    positive, and Lemma~\ref{lem:eigenvalue-transp-shuffle} together with
    \(x\le(\log n)^2\) gives
    \[
        |\alpha_\lambda-u|=O(n^{-2+o(1)}),\qquad
        \max\{\alpha_\lambda,u\}
        \le \exp\!\left(-\frac{2x}{n}+O(n^{-2+o(1)})\right).
    \]
    Since \(t=\Theta(n\log n)\), the power-difference identity yields
    \[
        |\alpha_\lambda^t-u^t|
        \le t|\alpha_\lambda-u|
        \max\{\alpha_\lambda,u\}^{t-1}
        \le e^{-2tx/n}n^{-1+o(1)}.
    \]
    Consequently, at \(b = (n-x)/n\), the squared difference satisfies
    \begin{equation} \label{eq:b=n-x/n-squared-diff}
        \left(\alpha_\lambda^t-\left(\frac{n-x}{n}\right)^s\right)^2
        \le e^{\frac{-4tx}{n}}n^{-2+o(1)}, 
    \end{equation}
    whereas at \(b = -x/n\), the left boundary, the squared difference is at most \(\exp(-4tx/n)(1+n^{-1+o(1)})\).
    Thus, we obtain a uniform bound \[\max_{-x/n\le b \le(n-x)/n}\bigl(\alpha_\lambda^t-b^s\bigr)^2 \leq \exp\bigl(-4tx/n\bigr)\bigl(1 + n^{-1+o(1)}\bigr).\]

    Using Lemma~\ref{lem:bounding-di-vals} and the previous bound on the squared difference of \(\alpha_\lambda^t\) and \(\beta_{\lambda^{(i)}}^s\), we can write our sum as the following:
    \begin{equation} \label{eq:bounding-L-sum}
        \sum_{\lambda \in \cL}  d_\lambda \sum_{i\in I(\lambda); i > 1} d_{\lambda^{(i)}}\, \bigl(\alpha_\lambda^{\,t}-\beta_{\lambda^{(i)}}^{\,s}\bigr)^2 
        \leq \sum_{\lambda\in\mathcal L}d_\lambda^2\sum_{\substack{i\in I(\lambda) \\ i>1}}\frac{i\lambda_i}{n}\,e^{-4tx/n}\bigl(1+n^{-1+o(1)}\bigr).
    \end{equation}

    Let 
    \[S_{\mathcal L}^{>1}:=\sum_{\lambda\in\mathcal L}d_\lambda^2\sum_{\substack{i\in I(\lambda) \\ i>1}} \frac{i\lambda_i}{n}\,e^{-4tx/n} \bigl(1+n^{-1+o(1)}\bigr), \qquad x=n-\lambda_1. \]
    Since \(\lambda\in\mathcal L\), we have \(x=n-\lambda_1\le (\log n)^2\). Thus all boxes outside the first row are contained in at most \(x\) lower rows, and each lower row has length at most \(x\). In particular, for \(i>1\),
    \[i\le (\log n)^2,\qquad \lambda_i\le(\log n)^2, \qquad |I(\lambda)|\le(\log n)^2.\]
    This gives
    \[\sum_{\substack{i\in I(\lambda)\\ i>1}}\frac{i\lambda_i}{n} \le \frac{(\log n)^6}{n}.\]
    Therefore,
    \begin{equation} \label{eq:start-of-reduction}
        S_{\mathcal L}^{>1} \le \frac{(\log n)^6}{n} \sum_{\lambda\in\mathcal L} d_\lambda^2 e^{-4tx/n} \bigl(1+n^{-1+o(1)}\bigr).
    \end{equation}
    
    Next, we use the dimension bound from Lemma~\ref{lem:bounds-d-lambda}. Namely, for \(x=n-\lambda_1\), \(d_\lambda^2\le \binom{n}{x}^2x! \le \frac{n^{2x}}{x!}. \)
    Grouping partitions according to the value of \(x=n-\lambda_1\), and using that the number of possibilities for the remaining shape is at most the number of partitions of \(x\), \(p(x)\), we obtain
    \[S_{\mathcal L}^{>1} \le \frac{(\log n)^6}{n} \sum_{x<(\log n)^2} p(x)\frac{n^{2x}}{x!}e^{-4tx/n} \bigl(1+n^{-1+o(1)}\bigr).\]
    
    Since \(t=\lfloor n\log n/2\rfloor+t'\), 
    \[n^{2x}e^{-4tx/n} = e^{-4t'x/n}\bigl(1+n^{-1+o(1)}\bigr).\]
    Using the Hardy--Ramanujan bound \(p(x)\le e^{\pi\sqrt{2x/3}}\), this gives
    \[S_{\mathcal L}^{>1} \le \frac{(\log n)^6}{n} \sum_{x<(\log n)^2} \frac{e^{\pi\sqrt{2x/3}-4t'x/n}}{x!} \bigl(1+n^{-1+o(1)}\bigr). \]
    Finally, we use the crude bound \(e^{\pi\sqrt{2x/3}}\le e^{3x}\):
    \begin{equation}\label{eq:end-of-reduction}
        S_{\mathcal L}^{>1} \le \frac{(\log n)^6}{n} \sum_{x<(\log n)^2} \frac{(e^{3-4t'/n})^x}{x!} \bigl(1+n^{-1+o(1)}\bigr).
    \end{equation}
    Letting \(\gamma=\exp(3-4t'/n),\) we have
    \[S_{\mathcal L}^{>1} \le \frac{(\log n)^6}{n} \bigl(1+n^{-1+o(1)}\bigr)e^\gamma.\]
    Moreover, under our assumptions on \(t'\), \(\gamma=\exp(3-4t'/n)=n^{o(1)}\) so \(e^\gamma=n^{o(1)}\). Therefore,
    \[S_{\mathcal L}^{>1} \le \frac{(\log n)^6}{n}\,n^{o(1)} = n^{-1+o(1)}.\]

\end{proof}

We now handle the contribution from the first-row corner. In this case the relevant random- and star-transposition eigenvalues are both close to \(1\), but the choice \(s=2t\) makes their powers agree to sufficiently high precision.
\begin{lem}\label{lem:L-i=1}
    Let $t=\lfloor \tfrac{n\log n}{2}\rfloor+t'$ be an integer with $|t'|\le n(\tfrac{\log\log n}{4}-\log\log\log n)$ and let $s=2t$.
    Then, \[\sum_{\lambda \in \cL} \;d_\lambda\;d_{\lambda^{(1)}}\, \bigl(\alpha_\lambda^{\,t}-\beta_{\lambda^{(1)}}^{\,s}\bigr)^2 \leq n^{-2+o(1)}.\]
\end{lem}

\begin{proof}
    In this case,
    \(\beta_{\lambda^{(1)}}=\lambda_1/n=(n-x)/n\). Therefore,
    \[
        \sum_{\lambda \in \cL}d_\lambda d_{\lambda^{(1)}}
        \bigl(\alpha_\lambda^{\,t}-\beta_{\lambda^{(1)}}^{\,s}\bigr)^2
        \le \sum_{\lambda \in \cL}d_\lambda^2
        \bigl(\alpha_\lambda^{\,t}-(\tfrac{n-x}{n})^s\bigr)^2.
    \]
    As shown in Equation~\eqref{eq:b=n-x/n-squared-diff}, the squared
    difference is at most
    \(\exp(-4tx/n)n^{-2+o(1)}\). Following the reductions from
    Equation~\eqref{eq:start-of-reduction} to
    Equation~\eqref{eq:end-of-reduction}, we obtain
    \[
        \sum_{\lambda \in \cL}d_\lambda^2
        \bigl(\alpha_\lambda^{\,t}-(\tfrac{n-x}{n})^s\bigr)^2
        \le n^{-2+o(1)}
        \sum_{x<(\log n)^2}\frac{e^{3x-4t'x/n}}{x!}
        =n^{-2+o(1)}.
    \]
\end{proof}

The previous two lemmas control two parts of the \(\mathcal L\)-sum: the first-row corner (Lemma~\ref{lem:L-i=1}) and the lower-row corners (Lemma~\ref{lem:L-i>1}). Combining these estimates gives the following bound for the full \(\mathcal L\) regime.  
\begin{lem}[Contribution of the $\mathcal L$ regime]\label{lem:L-regime}
    Under the assumptions of Lemmas~\ref{lem:L-i=1} and~\ref{lem:L-i>1}, we have
    \[\sum_{\lambda\in\mathcal L} d_\lambda \sum_{i\in I(\lambda)} d_{\lambda^{(i)}}
    \bigl(\alpha_\lambda^{\,t}-\beta_{\lambda^{(i)}}^{\,s}\bigr)^2
    \;\le\; n^{-1+o(1)}.\]
\end{lem}

\begin{proof}
    Split the sum into the contribution from $i=1$ and from $i>1$, and apply Lemmas~\ref{lem:L-i=1} and~\ref{lem:L-i>1}.
\end{proof}

\subsubsection{Bounding the \texorpdfstring{$\mathcal{L}'$}{L-prime} regime}
In the \(\mathcal L'\) regime, we have \(\lambda'_1 \ge n-(\log n)^2\), i.e.\ \(\lambda\) has a long first column. In this regime the relevant eigenvalues are close to \(-1\), and this is where the parity assumption on \(t\) enters: since \(s=2t\) is even, taking \(t\) even ensures that the signs of the corresponding eigenvalue powers are compatible. By transposing partitions and using the partial symmetry of the spectral data under \(\lambda\mapsto\lambda'\), we reduce the contribution of \(\mathcal L'\) to the corresponding bound for the \(\mathcal L\) regime.

We next record a simple estimate that will be used in the transposed regime.
\begin{lem}\label{lem:max-L'-bound-inner-sum}
    Let $t=\lfloor \tfrac{n\log n}{2}\rfloor+t'$ be an integer with
    $|t'|\le n(\tfrac{\log\log n}{4}-\log\log\log n)$, let
    $0\le x\le(\log n)^2$, and suppose, uniformly over this range of $x$,
    \[
        A=\exp\!\left(-\frac{2t(1+x)}{n}\right)
        \bigl(1+O(n^{-1+o(1)})\bigr).
    \]
    Uniformly for
    \(b\in[(2+x-n)/n,(2+x)/n]\),
    \[
        (A-b^{2t})^2
        \le \exp\!\left(-\frac{4t(1+x)}{n}\right)
        \bigl(1+n^{-1+o(1)}\bigr).
    \]
\end{lem}

\begin{proof}
    We maximize $(A-b^{2t})^2$ over the indicated interval. Its only
    possible interior critical values occur at $b=0$ or when $b^{2t}=A$.
    At a point of the latter type the squared difference is zero, while at
    $b=0$ it is
    \[
        A^2=\exp\!\left(-\frac{4t(1+x)}{n}\right)
        \bigl(1+n^{-1+o(1)}\bigr).
    \]
    At the left endpoint \(\frac{2+x-n}{n}\), we get 
    \[\left(\exp\left(\tfrac{-2t(1+x)}{n}\right)(1 + n^{-1+o(1)}) - \exp\left(\tfrac{-2t(2+x)}{n}\right)(1 + n^{-1+o(1)})\right)^2 = \exp\left(\tfrac{-4t(1+x)}{n}\right)(1 + n^{-1+o(1)}).\] 
    At the right endpoint \(\tfrac{2+x}{n}\), we have \((\tfrac{2+x}{n})^{2t} = \exp(-\Theta(n(\log n)^2))\), giving again 
    \[\exp\left(\tfrac{-4t(1+x)}{n}\right)(1 + n^{-1+o(1)}).\] 
    Therefore, \begin{equation*}
        \max_{b\in[\frac{2+x-n}{n}, \frac{2+x}{n}]}(A - b^{2t})^2 \leq \exp\left(\frac{-4t(1+x)}{n}\right)(1 + n^{-1+o(1)}).
    \end{equation*}
\end{proof}

\begin{lem}[Contribution of the $\mathcal L'$ regime]\label{lem:L'-regime}
    For even $t=\lfloor \tfrac{n\log n}{2}\rfloor+t'$ with $|t'|\le n(\tfrac{\log\log n}{4}-\log\log\log n)$ and $s=2t$, we have
    \[\sum_{\lambda \in \cL'} \;d_\lambda\;\sum_{i\in I(\lambda)} d_{\lambda^{(i)}}\, \bigl(\alpha_\lambda^{\,t}-\beta_{\lambda^{(i)}}^{\,s}\bigr)^2 \leq n^{-2+o(1)}.\]
\end{lem}

\begin{proof}
    Let $\lambda\in\mathcal L'$, so $\lambda'_1\ge n-(\log n)^2$, and set $\mu:=\lambda'$. Then $\mu\in\mathcal L$ with $x:=n-\mu_1=n-\lambda'_1\le (\log n)^2$. The corners correspond under transposition: if $i\in I(\lambda)$ and $j:=\lambda_i$, then $j\in I(\mu)$.
    The eigenvalues satisfy transpose identities (Lemma~\ref{lem:eigenvalue-transp-shuffle} and Lemma~\ref{lem:eigenvalues-star-transp}):
    \[\alpha_\lambda=\frac{2}{n}-\alpha_\mu,
    \qquad
    \beta_{\lambda^{(i)}}=\frac{2}{n}-\beta_{\mu^{(j)}}.\]

    Since $\mu\in\mathcal L$, Lemma~\ref{lem:eigenvalue-transp-shuffle} gives
    $\alpha_\mu = 1-\frac{2x}{n}+O(n^{-2+o(1)})$, hence
    \[\alpha_\lambda = -1+\frac{2(1+x)}{n}+O(n^{-2+o(1)}). \]
    In particular, for even $t$,
    \[\alpha_\lambda^{\,t} = \exp\!\Big(-\frac{2t(1+x)}{n}\Big)\,(1+n^{-1+o(1)}).\]

    Applying Lemma~\ref{lem:eigenvalues-star-transp} to $\mu\in\mathcal L$ and using the identity $\beta_{\lambda^{(i)}}=\frac{2}{n}-\beta_{\mu^{(j)}}$, we have
    \[\frac{2+x-n}{n} \leq \beta_{\lambda^{(i)}} \leq \frac{2 + x}{n}.\] 
    Therefore, for \(\lambda \in \mathcal{L}'\), and for every admissible corner, Lemma~\ref{lem:max-L'-bound-inner-sum} gives
    \begin{equation}
        \max_{i \in I(\lambda)}(\alpha_{\lambda}^t - \beta_{\lambda^{(i)}}^s)^2 
        \leq \exp\left(\frac{-4t(1+x)}{n}\right)(1 + n^{-1+o(1)}).
    \end{equation}

    Finally, by the hook-length formula, \(d_{\lambda} = d_{\lambda'}\), and by Lemma~\ref{lem:branching-rule},
    \(d_{\lambda^{(i)}}=d_{(\lambda')^{(\lambda_i)}}\). Using the bijection $(\lambda,i)\leftrightarrow(\mu=\lambda',\,j=\lambda_i)$ between corners under transposition, we obtain
    \begin{align*}
    \sum_{\lambda\in\mathcal L'} d_\lambda \sum_{i\in I(\lambda)} d_{\lambda^{(i)}}
    \bigl(\alpha_\lambda^{\,t}-\beta_{\lambda^{(i)}}^{\,s}\bigr)^2
    &\;\le\;
    (1 + n^{-1+o(1)}) \sum_{\mu \in \cL} d_\mu \sum_{j \in I(\mu)} d_{\mu^{(j)}} \exp\left(\tfrac{-4t(1+x)}{n}\right)\\
    &\;\le\;
    e^{-4t/n}(1+n^{-1+o(1)})\sum_{\mu\in\mathcal L} d_\mu \sum_{j\in I(\mu)} d_{\mu^{(j)}}
    \exp\!\Big(-\frac{4tx}{n}\Big).
    \end{align*}
    By the branching rule, the remaining sum is at most
    \[
        \sum_{\mu\in\mathcal L}d_\mu^2
        \exp\!\left(-\frac{4tx}{n}\right).
    \]
    Grouping by $x=n-\mu_1$, using
    $d_\mu^2\le n^{2x}/x!$ and
    $p(x)\le e^{3x}$, gives
    \begin{align*}
        \sum_{\mu\in\mathcal L}d_\mu^2e^{-4tx/n}
        &\le (1+n^{-1+o(1)})
        \sum_{0\le x\le(\log n)^2}
        \frac{e^{(3-4t'/n)x}}{x!}\\
        &\le \exp\!\bigl(e^{3-4t'/n}\bigr)n^{o(1)}
        =n^{o(1)}.
    \end{align*}
    Since $e^{-4t/n}=n^{-2+o(1)}$, the $\mathcal L'$ contribution is
    $n^{-2+o(1)}$, as claimed.
\end{proof}

\subsubsection{Combining the regime bounds}
\begin{proof}[ of Theorem~\ref{thm:even-time-comparison-thm}]
    By Corollary~\ref{cor:spectral-comparison-Sn}, for any $\sigma\in\mathfrak S_n$,
    \[4\|P^t(\sigma,\cdot)-Q^{s}(\sigma,\cdot)\|_{\mathrm{TV}}^2
    \le
    \sum_{\lambda\vdash n} d_\lambda \sum_{i\in I(\lambda)} d_{\lambda^{(i)}}\bigl(\alpha_\lambda^{\,t}-\beta_{\lambda^{(i)}}^{\,s}\bigr)^2. \]
    We decompose the right-hand side into the $\mathcal S$, $\mathcal L$, and $\mathcal L'$ regimes (see~\eqref{eq:def-S-L-L'}). The corresponding contributions are bounded in Lemmas~\ref{lem:S-regime},~\ref{lem:L-regime}, and~\ref{lem:L'-regime}, yielding
    \[4\|P^t(\sigma,\cdot)-Q^{s}(\sigma,\cdot)\|_{\mathrm{TV}}^2 \le n^{-1+o(1)}.\]
    Taking square roots completes the proof.
\end{proof}

\subsection{Removing the parity constraint} \label{subsection:removing-parity-constraint}
We record the stability estimates that remove the even-time restriction: first for \(\nu_t\), then for the star walk \(Y_s\).

To apply the results of \cite{jain2024hitting}, we retain their notation (see \eqref{eq:def-t-n}, \eqref{eq:def-gamma-n-t}, Definition~\ref{def:nu-t}). In particular, recall that
\[ t_n=\lfloor n\log n/2\rfloor,\qquad t=t_n+nc_n(t), \]
with \(nc_n(t)\in\mathbb Z\), and that
\[ \gamma_{n,t}=e^{-2(t-t_n)/n}=e^{-2c_n(t)}. \]
Also recall that \(\nu_t\) is the measure obtained by first sampling the number of forced fixed points from a truncated \(\mathrm{Pois}(\gamma_{n,t})\) distribution, then choosing this set uniformly, and finally placing a uniform permutation on its complement. Thus, within the cutoff window, the dependence of \(\nu_t\) on \(t\) is captured by the Poisson parameter \(\gamma_{n,t}\). 

\begin{thm}[{\cite[Theorem 1.3]{jain2024hitting}}] \label{thm:discrete-time-approx-random-trans}
    If \(|c_n(t)| \leq (\log \log n)/4 - \log \log \log n \), then \[d_{\mathrm{TV}}(X_t, \nu_t) \leq n^{-1 + o(1)}.\]
\end{thm}

We next need a continuity estimate for the approximating measures \(\nu_t\) as the time parameter varies. Since \(\nu_t\) depends on \(t\) through the Poisson parameter \(\gamma_{n,t}\), the following lemma gives a Lipschitz-type bound in \(\gamma_{n,t}\).

\begin{lem} \label{lem:comparison-poisson-general}
    With the above notation, fix a constant \(C>0\). For times \(t,\widetilde t\) satisfying
    \[
        |c_n(t)|,|c_n(\widetilde t)|
        \le \frac{\log\log n}{4}-\log\log\log n+\frac{C}{n},
    \]
    we have
    \[
        \|\nu_t-\nu_{\widetilde t}\|_{\mathrm{TV}}
        \le |\gamma_{n,t}-\gamma_{n,\widetilde t}|+n^{-\omega(1)}.
    \]
\end{lem}

\begin{proof}
    For \(0\le m\le n\), let \(\mu_m\) denote the law obtained by
    choosing \(m\) forced fixed points uniformly and placing a uniform
    permutation on their complement. The component measures \(\mu_m\) do not depend on
    time. We use the following elementary mixture bound: if \(M\) and
    \(\widetilde M\) have laws \(p\) and \(q\), respectively, and \(\mu_m\)
    is a common family of conditional laws, then
    \[
        d_{\mathrm{TV}}\left(\sum_m p_m\mu_m,\sum_m q_m\mu_m\right)
        \leq d_{\mathrm{TV}}(p,q).
    \]
 Applying this
    bound to the mixture representations of \(\nu_t\) and
    \(\nu_{\widetilde t}\) gives
    \[
        \|\nu_t-\nu_{\widetilde t}\|_{\mathrm{TV}}
        \le d_{\mathrm{TV}}\bigl(\mathcal L(M_t),
        \mathcal L(M_{\widetilde t})\bigr).
    \]

    We first compare untruncated Poisson laws. If
    \(0\le\gamma_1\le\gamma_2\), take independent random variables
    \(N_1\sim\mathrm{Pois}(\gamma_1)\) and
    \(K\sim\mathrm{Pois}(\gamma_2-\gamma_1)\). Then
    \(N_1+K\sim\mathrm{Pois}(\gamma_2)\), and hence
    \[
        d_{\mathrm{TV}}\bigl(\mathrm{Pois}(\gamma_1),
        \mathrm{Pois}(\gamma_2)\bigr)
        \le\mathbb P(K>0)
        =1-e^{-(\gamma_2-\gamma_1)}
        \le\gamma_2-\gamma_1.
    \]
    The same conclusion follows after interchanging \(\gamma_1\) and
    \(\gamma_2\).

    If \(N_\gamma\sim\mathrm{Pois}(\gamma)\), then
    \[
        d_{\mathrm{TV}}\bigl(\mathcal L(N_\gamma),
        \mathcal L(N_\gamma\mid N_\gamma\le n)\bigr)
        =\mathbb P(N_\gamma>n).
    \]
    Throughout the stated window, \(\gamma=n^{o(1)}\), and the standard
    Poisson tail estimate gives
    \[
        \mathbb P(N_\gamma>n)
        \le \left(\frac{e\gamma}{n+1}\right)^{n+1}
        =n^{-\omega(1)}.
    \]
    Applying the triangle inequality to the two truncated laws proves the
    claim.
\end{proof}

As an immediate consequence, consecutive times within the cutoff window give approximating
measures that differ by only \(n^{-1+o(1)}\). 
\begin{cor} \label{cor:one-step-comparison-poisson}
    With the above notation, suppose \(|t-\widetilde t|=1\) and both
    times satisfy the window in Lemma~\ref{lem:comparison-poisson-general}.
    Then
    \[
        \|\nu_t-\nu_{\widetilde t}\|_{\mathrm{TV}}
        \le n^{-1+o(1)}.
    \]
\end{cor}

\begin{proof}
    The two Poisson parameters differ by at most
    \[
        \max\{\gamma_{n,t},\gamma_{n,\widetilde t}\}
        (e^{2/n}-1)=n^{-1+o(1)}.
    \]
    Apply Lemma~\ref{lem:comparison-poisson-general}.
\end{proof}

We now prove the analogous stability estimate for the star walk; the proof reuses the same spectral decomposition as in \textsection~\ref{subsection:reduction-to-three-regimes} with a different weight.

\begin{lem} \label{lem:one-step-comparison-star}
    Let \(s\in\mathbb Z\) satisfy
    $
        |c_n^\star(s)|
        \le \frac12\log\log n-2\log\log\log n+\frac5n.
    $
    Then
    \[\|Y_s-Y_{s+1}\|_{\mathrm{TV}}\le n^{-1+o(1)}.\]
\end{lem}

\begin{proof}
    To bound \(4\|Q^s(x,\cdot)-Q^{s+1}(x,\cdot)\|_{\mathrm{TV}}^2\), where \(Q\) is the star-transposition kernel, we use the same eigenvalue indexing and partition decomposition as in \textsection~\ref{subsection:reduction-to-three-regimes}. In particular, the nontrivial eigenvalues are indexed by pairs \((\lambda,\lambda^{(i)})\) with \(\lambda\vdash n\) and \(i\in I(\lambda)\), where \(I(\lambda)\) denotes the set of row indices whose rightmost cell is a corner. We decompose the resulting sum according to \(\lambda\in\mathcal L\), \(\mathcal L'\), or \(\mathcal S\) (see \eqref{eq:def-S-L-L'}). With this notation, we consider
    \[ \sum_{\lambda\vdash n} d_\lambda \sum_{i\in I(\lambda)} d_{\lambda^{(i)}} \beta_{\lambda^{(i)}}^{\,2s}\bigl(1-\beta_{\lambda^{(i)}}\bigr)^2, \] 
    which we analyze separately on each of the three regimes.
    
    We begin with the contribution from \(\mathcal S\), the set of partitions \(\lambda\vdash n\) such that \(\lambda_1<n-(\log n)^2\) and \(\lambda_1'<n-(\log n)^2\).
    Since \(|\beta_{\lambda^{(i)}}|\le 1\), we have \(\bigl(1-\beta_{\lambda^{(i)}}\bigr)^2\le 4\). Therefore, 
    \[ \sum_{\lambda \in \mathcal S} d_\lambda \sum_{i\in I(\lambda)} d_{\lambda^{(i)}} \beta_{\lambda^{(i)}}^{\,2s}\bigl(1-\beta_{\lambda^{(i)}}\bigr)^2 
    \le 4\sum_{\lambda \in \mathcal S} d_\lambda \sum_{i\in I(\lambda)} d_{\lambda^{(i)}} \beta_{\lambda^{(i)}}^{\,2s}.\]
    Using the pairing argument in Lemma~\ref{lem:pairing-betas}, we reduce to considering the sum over \(J(\lambda)\) as defined in \eqref{eq:J(lambda)}:
    \[4\sum_{\lambda \in \mathcal S} d_\lambda \sum_{i\in I(\lambda)} d_{\lambda^{(i)}} \beta_{\lambda^{(i)}}^{\,2s} \leq 8\sum_{\lambda \in \mathcal S} d_\lambda \sum_{i\in J(\lambda)} d_{\lambda^{(i)}} \beta_{\lambda^{(i)}}^{\,2s}. \] 
    The right-hand side is a constant multiple of the quantity bounded in Lemma~\ref{lem:bounding-d-di-beta-2s}, and is \(n^{-\omega(1)}\).

    We now bound the $\mathcal L$ contribution. Recall that in this regime $\lambda_1=n-x$ with $x\le (\log n)^2$. We split according to whether the removed corner lies in the first row ($i=1$) or below the first row ($i>1$). The case $i>1$ is negligible because the corresponding star eigenvalues are small: for $i>1$,
    \begin{equation}\label{eq:beta-small}
        |\beta_{\lambda^{(i)}}|\;\le\;\frac{x}{n}.
        \end{equation}
    Indeed, $\beta_{\lambda^{(i)}}=\frac{1}{n}(\lambda_i-i+1)$. Since $i>1$, we have $\lambda_i\le \lambda_2\le x$ and, because there are only $x$ boxes below the first row, $i\le x+1$. Hence $|\lambda_i-i+1|\le x$.

    Using \eqref{eq:beta-small}, the branching rule $\sum_{i\in I(\lambda)} d_{\lambda^{(i)}}=d_\lambda$ (Lemma~\ref{lem:branching-rule}), and the crude bound $(1-\beta_{\lambda^{(i)}})^2\le 4$, we obtain
    \[\sum_{\lambda\in\mathcal L} d_\lambda \sum_{\substack{i\in I(\lambda)\\ i>1}} d_{\lambda^{(i)}}
    \,\beta_{\lambda^{(i)}}^{\,2s}\,(1-\beta_{\lambda^{(i)}})^2
    \;\le\;
    4\sum_{\lambda\in\mathcal L} d_\lambda^2\left(\frac{x}{n}\right)^{2s}.\]
    
    We now sum over $\lambda$ by grouping according to $x=n-\lambda_1$. For fixed $x$, we have the dimension bound
    $d_\lambda^2\le \binom{n}{x}^2\,x!\le n^{2x}/x!$, and the number of partitions with $\lambda_1=n-x$ is at most $p(x)$.
    Using the Hardy--Ramanujan bound $p(x)\le \exp(\pi\sqrt{2x/3})\le e^{3x}$ for $x\le (\log n)^2$, we get
    \[
    4\sum_{\lambda\in\mathcal L} d_\lambda^2\left(\frac{x}{n}\right)^{2s}
    \;\le\;
    4\sum_{1\le x\le (\log n)^2} p(x)\,\frac{n^{2x}}{x!}\left(\frac{x}{n}\right)^{2s}
    \;\le\;
    4\sum_{1\le x\le (\log n)^2} \frac{e^{3x}n^{2x}}{x!}\left(\frac{x}{n}\right)^{2s}.
    \]
    For each $1\le x\le (\log n)^2$, the factor $(x/n)^{2s}$ dominates: since $s=n\log n+O(n\log\log n)$,
    \[
    \left(\frac{x}{n}\right)^{2s}
    =
    \exp\bigl(2s(\log x-\log n)\bigr)
    \le
    \exp\bigl(-(2+o(1))n(\log n)^2\bigr),\]
    uniformly over $x\le (\log n)^2$. The remaining factor $\frac{e^{3x}n^{2x}}{x!}$ is at most $\exp(O((\log n)^3))$ in this range,
    so each summand is $\exp(-(2+o(1))n(\log n)^2)$. Summing over at most $(\log n)^2$ values of $x$ yields
    that the $\mathcal L$ contribution with $i>1$ is $n^{-\omega(1)}$.

    We now turn to the contribution from $\mathcal L$ with $i=1$. Write $\lambda_1=n-x$ with $0\le x\le(\log n)^2$. By Lemma~\ref{lem:eigenvalues-star-transp},
    \[\beta_{\lambda^{(1)}}=\frac{1}{n}(\lambda_1-1+1)=\frac{\lambda_1}{n}=1-\frac{x}{n}.\]
    Our goal is to bound
    \[\sum_{\lambda\in\mathcal L} d_\lambda\,d_{\lambda^{(1)}}\, \beta_{\lambda^{(1)}}^{\,2s}\bigl(1-\beta_{\lambda^{(1)}}\bigr)^2.\]
    Using $d_{\lambda^{(1)}}\le d_\lambda$ (branching rule) and the standard bound $d_\lambda^2\le \binom{n}{x}^2 x!\le n^{2x}/x!$, and grouping partitions by $x=n-\lambda_1$, we obtain
    \[\sum_{\lambda\in\mathcal L} d_\lambda d_{\lambda^{(1)}}
    \beta_{\lambda^{(1)}}^{\,2s}\bigl(1-\beta_{\lambda^{(1)}}\bigr)^2
    \;\le\;
    \sum_{x\le(\log n)^2} p(x)\,\frac{n^{2x}}{x!}\,
    \left(1-\frac{x}{n}\right)^{2s}\left(\frac{x}{n}\right)^2,\]
    where $p(x)$ is the partition function. Using $p(x)\le e^{3x}$ for $x\le(\log n)^2$, it remains to control the two factors $\left(1-\frac{x}{n}\right)^{2s}$ and $\left(\frac{x}{n}\right)^2$.
    
    For the first, using $s=s_n+nc_n^\star(s)$ and the assumed window,
    \[\left(1-\frac{x}{n}\right)^{2s} 
    =\exp\!\left(-\frac{2xs}{n}\right)\bigl(1+n^{-1+o(1)}\bigr)
    \le n^{-2x}\,(\log n)^x\,(\log\log n)^{-4x}\bigl(1+n^{-1+o(1)}\bigr).
    \]
    Moreover, since $x\le(\log n)^2$,
    \[\left(\frac{x}{n}\right)^2 \le \frac{(\log n)^4}{n^2}=n^{-2+o(1)}.\]
    Combining these bounds yields
    \[\sum_{x\le(\log n)^2} p(x)\,\frac{n^{2x}}{x!}\,\left(1-\frac{x}{n}\right)^{2s}\left(\frac{x}{n}\right)^2
    \;\le\;
    n^{-2+o(1)}\sum_{x\le(\log n)^2}\frac{e^{3x}}{x!}
    \left(\frac{\log n}{(\log\log n)^4}\right)^x.\]
    The remaining sum is $n^{o(1)}$ (indeed it is dominated by $\exp\bigl(e^3\frac{\log n}{(\log\log n)^4}\bigr)$), so the $i=1$ contribution is $n^{-2+o(1)}$.

    We now turn to the contribution from $\mathcal L'$. Let $\lambda\in\mathcal L'$ and set $\mu:=\lambda'$, so that $\mu\in\mathcal L$ and $y:=n-\mu_1=n-\lambda'_1\le (\log n)^2$. As before, we transfer eigenvalue estimates to the conjugate partition $\mu$, whose first row is large.
    For $i\in I(\lambda)$, writing $j:=\lambda_i\in I(\mu)$, the transpose relation for star eigenvalues gives
    \[\beta_{\lambda^{(i)}}=\frac{2}{n}-\beta_{\mu^{(j)}}.\]
    We split into the cases $j>1$ and $j=1$.
    
    First consider $j>1$. Then by the $\mathcal L$-analysis we have $|\beta_{\mu^{(j)}}|\le y/n$. Consequently,
    \[|\beta_{\lambda^{(i)}}| =\left|\frac{2}{n}-\beta_{\mu^{(j)}}\right| \le \frac{2+y}{n}. \]
    Since
    \[
        \left(\frac{2+y}{n}\right)^{2s}
        \le \exp\bigl(-(2+o(1))n(\log n)^2\bigr),
    \]
    the $j>1$ contribution is bounded exactly as in the corresponding $\mathcal L$ case.
    
    Next consider $j=1$. In this case,
    \[
        \beta_{\lambda^{(\lambda'_1)}}
        =\frac{2}{n}-\beta_{\mu^{(1)}}
        =-\left(1-\frac{2+y}{n}\right).
    \]
    Here the factor $(1-\beta)^2$ is bounded by $4$, rather than supplying the factor $n^{-2+o(1)}$ used in the $\mathcal L$ estimate. Nevertheless, the shift by \(2/n\) in the eigenvalue gives a stronger exponential decay. By transposition, the relevant corner dimension is $d_{\mu^{(1)}}$, and hence this contribution is at most
    \begin{align*}
        4\sum_{\mu\in\mathcal L}d_\mu d_{\mu^{(1)}}
        \left(1-\frac{2+y}{n}\right)^{2s}
        &\le 4e^{-4s/n}
        \sum_{y\le(\log n)^2}\frac{e^{3y}}{y!}
        n^{2y}e^{-2sy/n}\\
        &\le n^{-4+o(1)}
        \sum_{y\le(\log n)^2}
        \frac{\exp((3-2c_n^\star(s))y)}{y!}
        =n^{-4+o(1)}.
    \end{align*}
    In the last step we used
    \[
        \exp(3-2c_n^\star(s))
        \le \frac{e^3\log n}{(\log\log n)^4}\,e^{O(1/n)},
    \]
    so the exponential generating function on the preceding line is $n^{o(1)}$.
    
    Combining the bounds for \(\mathcal S\), \(\mathcal L\), and \(\mathcal L'\),
    and taking the square root, we obtain
    \[\|Q^s(x,\cdot)-Q^{s+1}(x,\cdot)\|_{\mathrm{TV}} \le n^{-1+o(1)}, \] 
    as claimed.
\end{proof}

We now pass from the even-time comparison to the full deterministic-time approximation by matching each time in the cutoff window to a nearby even time.

\begin{proof}[ of Corollary~\ref{cor:final-tv-distance-star-walk}]
    In this proof, we combine the even-time comparison theorem (Theorem~\ref{thm:even-time-comparison-thm}), the deterministic-time approximation for random transpositions (Theorem~\ref{thm:discrete-time-approx-random-trans}), and the stability estimates for $\nu_t$ and $Y_s$ (Corollary~\ref{cor:one-step-comparison-poisson} and Lemma~\ref{lem:one-step-comparison-star}). 
    Since $t=\lceil s/2\rceil$,
    \[
        \left|c_n(t)-\frac12c_n^\star(s)\right|\le\frac1n.
    \]
    Write
    \[
        R_n:=\frac14\log\log n-\log\log\log n.
    \]
    Because \(|c_n(t)|\le R_n+1/n\) and changing time by one changes
    \(c_n\) by \(1/n\), there is an even integer \(\widehat t\) such that
    \[
        |\widehat t-t|\le2,\qquad |c_n(\widehat t)|\le R_n.
    \]
    Indeed, if \(t\) is odd, one of \(t-1,t+1\) works; if \(t\) is even,
    one may use \(t\) unless it lies outside the window, in which case
    shifting it two steps toward the window works. In particular,
    \(|2\widehat t-s|\le5\).

    The triangle inequality now gives
    \begin{align*}
        d_{\mathrm{TV}}(Y_s,\nu_t)
        &\le d_{\mathrm{TV}}(Y_s,Y_{2\widehat t})
        +d_{\mathrm{TV}}(Y_{2\widehat t},X_{\widehat t})
        +d_{\mathrm{TV}}(X_{\widehat t},\nu_{\widehat t})
        +d_{\mathrm{TV}}(\nu_{\widehat t},\nu_t).
    \end{align*}
    Every intermediate star-shuffle time differs from \(s\) by at most
    five, so Lemma~\ref{lem:one-step-comparison-star} bounds the first
    term by \(n^{-1+o(1)}\). The second term is \(n^{-1/2+o(1)}\) by
    Theorem~\ref{thm:even-time-comparison-thm}, and the third is
    \(n^{-1+o(1)}\) by Theorem~\ref{thm:discrete-time-approx-random-trans}.
    Finally, at most two applications of
    Corollary~\ref{cor:one-step-comparison-poisson} bound the last term by
    \(n^{-1+o(1)}\). Summing these bounds proves the result.
\end{proof}

With the deterministic-time approximation in hand, we now turn to the conditioning argument needed for hitting-time mixing.

\section{Untouched sets as an approximate sufficient statistic} \label{sec:approximate-sufficient-statistic}
The goal of this section is to show that, for times in the star-shuffle window, the untouched set is an approximate sufficient statistic: conditioned on a small untouched set \(T\) the law of the walk is close to the uniform law on permutations fixing \(T\). The notation and organization in this section follow Jain--Sawhney \cite{jain2024hitting} closely, but the reduced distributions and conditioning statements must be adapted to the star-transposition setting.

The proof estimates \(\mathcal G(T)\) and \(\mathcal F(T)\), uses inclusion--exclusion to condition on the exact untouched set, and compares the resulting law with the reduced-space approximation on \([n]\setminus T\).

\subsection{Untouched set notation}
We begin by introducing notation that will be used throughout this section. A superscript or subscript \(\star\) denotes the star-shuffle analogue of the corresponding object in \cite{jain2024hitting}. Let \((1,i_1),\ldots,(1,i_s)\) denote the (random) sequence of star transpositions chosen by the star-transposition walk up to time \(s\). For \(1\le M\le n\), write
\[
    [M]_\star:=\{2,\ldots,M\},
\]
with \([1]_\star=\varnothing\); in particular, \([n]_\star=\{2,\ldots,n\}\) is the set of non-hub elements. For a subset \(T\subseteq[n]_\star\), let \(\mathcal G(T)=\mathcal G^s(T)\) denote the event that
\[ \{i_1,\ldots,i_s\}\cap T =\varnothing, \]
that is, the walk does not touch any element of \(T\) up to and including time \(s\). 
Define \(\mathcal F(T)=\mathcal F^s(T)\) by
\[\mathcal F(T) = \mathcal G(T)\cap \bigcap_{j\in [n]_{\star}\setminus T} \mathcal G^{c}(\{j\}), \]
so that \(\mathcal F(T)\) is the event that the walk touches every element of \([n]_\star \setminus T\) and does not touch \(T\) by time \(s\). For convenience, when \(T=[M]_\star\), we write \(\mathcal G(M)\) and \(\mathcal F(M)\) in place of \(\mathcal G([M]_\star)\) and \(\mathcal F([M]_\star)\).

Recall that \(s_n=\lfloor n\log n\rfloor\), and write
\begin{equation}\label{eq:def-of-sn}
    s=s_n+nc_n^\star(s),\qquad nc_n^\star(s)\in\mathbb Z.
\end{equation}
Let
\begin{equation}\label{eq:def-of-gamma-n-s-star}
    \gamma_{n,s}^\star=e^{-(s-s_n)/n}=e^{-c_n^\star(s)},
\end{equation}
and set
\begin{equation}\label{eq:def-K-n}
    K_n:=\left\lceil\frac{\log n}{\log\log n}\right\rceil.
\end{equation}
Let \(\nu^\star_s\) be the distribution given by the following procedure. First, sample \(M_s \in \{0, 1, \ldots, n-1\}\) according to the distribution 
\[\bbP(M_s = m) = \frac{\bbP(\mathrm{Pois}(\gamma^\star_{n,s})=m)}{\bbP(\mathrm{Pois}(\gamma^\star_{n,s})\leq n-1)}.\]
Then, sample a random set \(S_s \subset [n]_\star\) of size \(M_s\) uniformly at random. The set \(S_s\) is the set of forced fixed points in \([n]_{\star}\). Finally, sample a random element of \(\mathfrak{S}_{[n] \setminus S_s}\) uniformly and view it as an element of \(\mathfrak{S}_n\) by fixing all the elements in \(S_s\).
If \(t=\lceil s/2\rceil\), then
\(\gamma_{n,t}=(1+O(n^{-1}))\gamma_{n,s}^\star\), since both
\(2\lceil s/2\rceil-s\) and \(s_n-2t_n\) belong to \(\{0,1\}\).
Thus \(\nu_s^\star\) is the non-hub analogue of \(\nu_t\), with the
Poisson parameter changed only by a relative \(O(n^{-1})\) factor.

\subsection{Estimates for \texorpdfstring{\(\mathcal G(T)\) and \(\mathcal F(T)\)}{G(T) and F(T)}}\label{subsection:G-F-estimates}

The next lemma packages the probability estimates for the events \(\cG(T)\) and \(\cF(M)\) needed in the remainder of the argument. These estimates will later be combined with inclusion--exclusion to control the law of the star transposition shuffle at time \(s\) conditioned on a small untouched set. It is the star-transposition analogue of \cite[Lemma 2.2]{jain2024hitting}, with only minor changes to adapt to the present setting.
\begin{lem} \label{lem:relationship-f(T)-g(T)}
    Let \(s_n\), \(s\), \(\gamma_{n,s}^\star\), and \(c_n^\star(s)\) be
    defined as in equations \eqref{eq:def-of-sn} and
    \eqref{eq:def-of-gamma-n-s-star}. Let \(1\le M_n\le n\) and let
    \(T_n \subseteq [n]_\star\) satisfy
    \([M_n]_\star\subseteq T_n\) and \(|T_n|=n^{o(1)}\). Suppose that
    \(|c_n^\star(s)| \leq (\log \log n)/2 - 4 \log \log \log n\). Then
    the following hold:
    \begin{enumerate}
        \item \(\bbP[\cG(T_n)] = (1 + O(n^{-1+o(1)})) \, \bbP[\cG(M_n)] \, (\gamma_{n, s}^{\star})^{|T_n| - M_n + 1} \, n^{M_n - 1 - |T_n|}\),
        \item \(\bbP[\cF(M_n)] = (1 + O(n^{-1+o(1)})) \bbP[\cG(M_n)] \exp(-\gamma_{n, s}^\star)\),
        \item for every \(L_n=n^{o(1)}\),
        \(\sum_{S \supseteq [M_n]_\star,\, |S| \leq L_n} \bbP[\cG(S)] \leq n^{o(1)} \bbP[\cF(M_n)]\).
    \end{enumerate}
\end{lem}

\begin{proof}
    Write \(M=M_n\), \(T=T_n\), and \(\gamma=\gamma_{n,s}^\star\). Uniformly for \(x=n^{o(1)}\),
    \[
        \left(1-\frac{x}{n}\right)^s
        =\bigl(1+O(n^{-1+o(1)})\bigr)\exp\!\left(-\frac{sx}{n}\right).
    \]
    Since \(|[M]_\star|=M-1\), applying this estimate first with \(x=|T|\) and then with \(x=M-1\) gives
    \begin{equation}\label{eq:G(T)-G(M)-relationship}
        \bbP(\cG(T))
        =\bigl(1+O(n^{-1+o(1)})\bigr)\bbP(\cG(M))
        \left(\frac{\gamma}{n}\right)^{|T|-M+1}.
    \end{equation}
    This proves the first claim.

    For the second claim, inclusion--exclusion gives the exact indicator identity
    \begin{equation}\label{eq:g-and-f-relationship-indicator}
        \mathbf 1_{\cF(M)}
        =\sum_{[M]_\star\subseteq T\subseteq[n]_\star}
        (-1)^{|T|-M+1}\mathbf 1_{\cG(T)}.
    \end{equation}
    Apply the Bonferroni inequalities to truncate this sum after \(2k+1\) terms, where \(k=\lfloor(\log n)^2\rfloor\). Taking expectations and using \eqref{eq:G(T)-G(M)-relationship}, the upper and lower bounds differ by at most
    \[
        \bigl(1+O(n^{-1+o(1)})\bigr)\bbP(\cG(M))
        \binom{n-M}{2k+1}\left(\frac{\gamma}{n}\right)^{2k+1}.
    \]
    In the stated window, \(\gamma\le\sqrt{\log n}\), so this error is \(n^{-\omega(1)}\bbP(\cG(M))\). Moreover, uniformly for \(j\le2k+1\),
    \[
        \binom{n-M}{j}n^{-j}
        =\frac{1+O(n^{-1+o(1)})}{j!}.
    \]
    The absolute sum of all approximation errors in the truncated
    inclusion--exclusion series, after division by \(\bbP(\cG(M))\), is at
    most
    \[
        n^{-1+o(1)}\sum_{j\le2k+1}\frac{\gamma^j}{j!}
        \le n^{-1+o(1)}e^\gamma.
    \]
    Since the target is \(e^{-\gamma}\) and \(e^{2\gamma}=n^{o(1)}\),
    this is a relative error \(n^{-1+o(1)}\). The omitted
    exponential-series tail and the Bonferroni gap are
    \(n^{-\omega(1)}e^{-\gamma}\). Thus the alternating sum equals
    \(\bigl(1+O(n^{-1+o(1)})\bigr)e^{-\gamma}\), proving the second
    claim without passing an uncontrolled relative error through the
    cancellation.

    Finally, group the sets in the third sum according to \(r=|S|-M+1\). Parts~(1) and~(2) imply
    \begin{align*}
        \sum_{\substack{S\supseteq[M]_\star\\ |S|\le L_n}}
        \bbP(\cG(S))
        &\le \bigl(1+O(n^{-1+o(1)})\bigr)\bbP(\cF(M))e^\gamma
        \sum_{r\ge0}\binom{n-M}{r}\left(\frac{\gamma}{n}\right)^r \\
        &\le \bbP(\cF(M))\exp(O(\gamma))
        =n^{o(1)}\bbP(\cF(M)),
    \end{align*}
    because \(\gamma\le\sqrt{\log n}\).
\end{proof}

\subsection{Inclusion--exclusion identity}\label{subsection:inclusion-exclusion}
For every positive integer \(m\), write \(s_m:=\lfloor m\log m\rfloor\).
For any \(T \subseteq [n]_\star\), define the auxiliary real parameter
\[s_T = s_{n - |T|} + (n - |T|) c_n^\star(s).\]
It is used only to parameterize the approximating measure, so no integrality assumption on \(s_T\) is needed.
Analogous to our definition of \(\gamma_{n, s}^\star\), we define 
\begin{equation} \label{eq:gamma-n-T-s-T}
\gamma_{n - |T|, s_T}^\star = \exp(-(s_{T} - s_{n-|T|})/(n-|T|)) = \exp(-c_n^\star(s)).
\end{equation}
Therefore, \(\gamma_{n - |T|, s_T}^\star = \gamma_{n, s}^\star\).
The following reduced approximation is the star-transposition analogue of
\cite[Definition 2.3]{jain2024hitting}.
\begin{defn} \label{def:nu-star-T}
    Let \(\nu^{\star, T} = \nu^{\star, T}_{s_T}\) denote the distribution on \(\mathfrak{S}_{[n] \setminus T}\), which is defined as the same distribution as \(\nu^\star_s\) except that time \(s\) is replaced by \(s_T\) and the ground set \([n]\) is replaced by \([n]\setminus T\). Using the natural inclusion \(\mathfrak{S}_{[n]\setminus T} \hookrightarrow \mathfrak{S}_n\), we view \(\nu^{\star, T}\) as a distribution on \(\mathfrak{S}_n\).

    In particular, we first sample \(M_{s_T} \in \{0, 1, \ldots, n-|T|-1\}\) according to the distribution 
    \[\bbP(M_{s_T} = m) = \frac{\bbP(\mathrm{Pois}(\gamma^\star_{n-|T|,s_T})=m)}{\bbP(\mathrm{Pois}(\gamma^\star_{n-|T|,s_T})\leq n-|T|-1)}.\]
    Then, sample a random set \(S_{s_T} \subset [n]_\star \setminus T\) of size \(M_{s_T}\) uniformly at random. Finally, sample a random element of \(\mathfrak{S}_{[n] \setminus (T \cup S_{s_T})}\) uniformly and view it as an element of \(\mathfrak{S}_n\) by fixing all the elements in \(T \cup S_{s_T}\). 
\end{defn}

The next lemma is the inclusion--exclusion step. It shows that the alternating sum of the terms \(\mathbb P(\mathcal G(T))\,\nu^{\star,T}(\sigma_n)\) collapses to the main term \(\mathbb P(\mathcal F(M))/(n-M+1)!\). It is the analogue of \cite[Lemma 2.4]{jain2024hitting}, with minor changes to adapt to the star-transposition setting.

\begin{lem}\label{lem:relate-f-to-g-dot-nu}
    Suppose that
    $
        |c_n^\star(s)|\le \frac12\log\log n-4\log\log\log n.$
    Let \(\mathrm{Fix}_{\star}(\sigma)\) denote the set of fixed points of \(\sigma\) excluding \(1\). For any permutation \(\sigma_n\in\mathfrak S_n\) and any integer \(1\le M_n\le n\) satisfying \(M_n=n^{o(1)}\), \([M_n]_{\star}\subseteq\mathrm{Fix}_{\star}(\sigma_n)\), and
    \[
        |\mathrm{Fix}_{\star}(\sigma_n)\setminus[M_n]_\star|\le K_n,
    \]
    we have
    \[\sum_{[M_n]_\star \subseteq T \subseteq \mathrm{Fix}_{\star}(\sigma_n)} (-1)^{|T|-M_n+1} \, \mathbb{P}(\mathcal{G}(T)) \, \nu^{\star, T}(\sigma_n) 
    = \left(1 + O(n^{-1/2+o(1)})\right) \frac{\mathbb{P}(\mathcal{F}(M_n))}{(n-M_n+1)!}.\]
\end{lem}

\begin{proof}
    Set
    \[
        b:=M-1,\qquad f:=|\mathrm{Fix}_\star(\sigma)|,\qquad
        D:=f-b\le K_n,\qquad \gamma:=\gamma_{n,s}^\star.
    \]
    For \(T\subseteq\mathrm{Fix}_\star(\sigma)\) containing \([M]_\star\), write
    \(q:=|T|-b\). By definition,
    \[\nu^{\star,T}(\sigma)
    = \sum_{r=0}^{D-q}
    \frac{\mathbb P(\mathrm{Pois}(\gamma)=r)}
    {\mathbb P(\mathrm{Pois}(\gamma)\le n-1-b-q)}
    \frac{\binom{D-q}{r}}{\binom{n-1-b-q}{r}}
    \frac{1}{(n-b-q-r)!}.\]
    The Poisson truncation probability is \(1-n^{-\omega(1)}\). Uniformly for
    \(b,q,r=n^{o(1)}\), the exact factorial ratio gives
    \[
        \frac{\mathbb P(\mathrm{Pois}(\gamma)=r)}
        {\mathbb P(\mathrm{Pois}(\gamma)\le n-1-b-q)
        \binom{n-1-b-q}{r}(n-b-q-r)!}
        =
        \left(1+O(n^{-1+o(1)})\right)
        \frac{e^{-\gamma}\gamma^r n^q}{(n-b)!}.
    \]
    Consequently,
    \begin{equation}\label{eq:rewrite-nu-star-T}
        \nu^{\star,T}(\sigma)
        =
        \left(1+O(n^{-1+o(1)})\right)
        \frac{e^{-\gamma}n^q}{(n-b)!}
        \sum_{r=0}^{D-q}\binom{D-q}{r}\gamma^r.
    \end{equation}

    Lemma~\ref{lem:relationship-f(T)-g(T)} gives, uniformly for \(q\le D\),
    \[
        \mathbb P(\mathcal G(T))
        =\left(1+O(n^{-1+o(1)})\right)
        \mathbb P(\mathcal G(M))\left(\frac{\gamma}{n}\right)^q
    \]
    and
    \[
        \mathbb P(\mathcal F(M))
        =\left(1+O(n^{-1+o(1)})\right)
        \mathbb P(\mathcal G(M))e^{-\gamma}.
    \]
    We cannot simply pull these relative errors through the alternating sum.
    Instead, summing their absolute contributions and using
    \eqref{eq:rewrite-nu-star-T} shows that the normalized accumulated error is at most
    \begin{align*}
        n^{-1+o(1)}
        \sum_{q=0}^{D}\binom Dq\gamma^q
        \sum_{r=0}^{D-q}\binom{D-q}{r}\gamma^r
        &=n^{-1+o(1)}(1+2\gamma)^D\\
        &\le n^{-1/2+o(1)}.
    \end{align*}
    For the last inequality, the window assumption gives
    \(\gamma\le\sqrt{\log n}/(\log\log n)^4\), while
    \(D\le K_n\); hence \((1+2\gamma)^D\le n^{1/2+o(1)}\).

    After removing this error, the normalized main term is
    \[
        \sum_{q=0}^{D}\binom Dq(-\gamma)^q(1+\gamma)^{D-q}
        =(1+\gamma-\gamma)^D=1.
    \]
    Multiplying back by
    \(\mathbb P(\mathcal F(M))/(n-b)!\) proves the claim.
\end{proof}

\subsection{Conditioned approximation on \texorpdfstring{\([n]\setminus T\)}{the reduced ground set}}\label{subsection:conditioned-approximation}
\begin{lem} \label{lem:TV-of-mixtures}
    Let $S$ be a random variable taking values in a finite set $\mathcal{S}$, and let $A$ and $B$ be random variables taking values in a common finite space $\Omega$. Suppose $A$ and $B$ are defined on the same probability space. Then
    \[d_{\mathrm{TV}}(\mathcal{L}(A), \mathcal{L}(B)) \le 
    \mathbb{E}\!\left[d_{\mathrm{TV}}\big(\mathcal{L}(A\mid S), \mathcal{L}(B\mid S)\big)\right].\]
\end{lem}

\begin{proof}
    Let \(p_s=\bbP(S=s)\), \(\mu_s=\mathcal L(A\mid S=s)\), and \(\nu_s=\mathcal L(B\mid S=s)\). By the law of total probability and the triangle inequality,
    \begin{align*}
        d_{\mathrm{TV}}(\mathcal L(A),\mathcal L(B))
        &=\frac12\sum_{x\in\Omega}\left|\sum_{s\in\mathcal S}p_s(\mu_s(x)-\nu_s(x))\right|\\
        &\le\sum_{s\in\mathcal S}p_s d_{\mathrm{TV}}(\mu_s,\nu_s),
    \end{align*}
    which is the claimed bound.
\end{proof}

The next lemma is the key self-reducibility step, corresponding to
\cite[Lemma 2.5]{jain2024hitting}. Conditioned on \(\mathcal G(T)\), the
star-transposition walk evolves only on \([n]\setminus T\). This allows us to
import the deterministic time approximation from
Section~\ref{sec:deterministic-time-approx-for-star} on the reduced space.

\begin{lem} \label{lem:self-reducibility}
    Let \(s\in\mathbb Z_{\ge0}\), let \(s_n=\lfloor n\log n\rfloor\), and suppose
    \[s=s_n+nc_n^\star(s), \qquad |c_n^\star(s)|\le \tfrac{1}{2}\log\log n-4\log\log\log n. \]
    Let \(T\subseteq [n]_{\star}\) with \(|T| \leq n^{o(1)}\), and let \(\widetilde{Y}_s^{\,T}\) denote the star transposition walk \(Y_s\) conditioned on the event \(\mathcal G(T)\). Let \(\nu^{\star,T}\) be the distribution defined in Definition~\ref{def:nu-star-T}. Then
    \[d_{\mathrm{TV}}\!\left(\mathcal L(\widetilde{Y}_s^{\,T}),\nu^{\star,T}\right) \le n^{-1/2+o(1)}.\]
\end{lem}

\begin{proof}
    Set \(m:=n-|T|\), \(u:=\lceil s/2\rceil\), and
    \(u_m:=\lfloor m\log m/2\rfloor\). Define
    \[
        \bar\gamma_{m,u}:=\exp\!\left(-\frac{2(u-u_m)}m\right).
    \]
    Let \(\bar\nu_{m,u}\) be the random-transposition mixture on
    \([n]\setminus T\) with parameter \(\bar\gamma_{m,u}\), and let
    \(\bar\nu^\star_{m,u}\) be the analogous mixture in which forced fixed
    points are selected only from the non-hub elements and the Poisson law is
    truncated at \(m-1\). We use
    \[
        d_{\mathrm{TV}}(\mathcal L(\widetilde Y_s^{\,T}),\nu^{\star,T})
        \le d_{\mathrm{TV}}(\mathcal L(\widetilde Y_s^{\,T}),\bar\nu_{m,u})
        +d_{\mathrm{TV}}(\bar\nu_{m,u},\bar\nu^\star_{m,u})
        +d_{\mathrm{TV}}(\bar\nu^\star_{m,u},\nu^{\star,T}).
    \]

    \smallskip
    \noindent\emph{Step 1: the reduced star walk.}
    Conditional on \(\mathcal G(T)\), every selected index is uniform on
    \([n]\setminus T\), so \(\widetilde Y_s^{\,T}\) is exactly the
    \(m\)-point star-transposition walk run for \(s\) steps. Its star-window
    coordinate satisfies
    \begin{align*}
        \left|\frac{s-\lfloor m\log m\rfloor}{m}\right|
        &\le |c_n^\star(s)|
        +\frac{1+|T|\bigl(|c_n^\star(s)|+1+\log n\bigr)}m\\
        &=|c_n^\star(s)|+n^{-1+o(1)},
    \end{align*}
    where we used the mean value theorem for \(x\log x\). Since
    \(\log\log m=\log\log n+o(1)\), the slack in the hypothesis places this
    reduced walk in the window of
    Corollary~\ref{cor:final-tv-distance-star-walk}. Therefore
    \[
        d_{\mathrm{TV}}(\mathcal L(\widetilde Y_s^{\,T}),\bar\nu_{m,u})
        \le n^{-1/2+o(1)}.
    \]

    \smallskip
    \noindent\emph{Step 2: removing the hub from the forced set.}
    Let \(N\sim\mathrm{Pois}(\bar\gamma_{m,u})\), and couple
    \(\bar K\sim(N\mid N\le m)\) and
    \(K^\star\sim(N\mid N\le m-1)\) maximally. Their laws differ by
    \(n^{-\omega(1)}\). Conditional on a common value \(k\le m-1\), the
    subset choices can differ only when the subset selected from all \(m\)
    points contains the hub, an event of probability \(k/m\). Hence
    \[
        d_{\mathrm{TV}}(\bar\nu_{m,u},\bar\nu^\star_{m,u})
        \le n^{-\omega(1)}
        +\frac{\mathbb E[N\mid N\le m]}{m}
        =n^{-1+o(1)}.
    \]

    \smallskip
    \noindent\emph{Step 3: matching the Poisson parameters.}
    Put
    \[
        \varepsilon_{m,s}:=
        2u-s+\lfloor m\log m\rfloor-2u_m\in\{0,1,2\}.
    \]
    Then
    \begin{align*}
        \Delta
        &:=\frac{2(u-u_m)}m-c_n^\star(s)\\
        &=\frac{s_n-\lfloor m\log m\rfloor
        +|T|c_n^\star(s)+\varepsilon_{m,s}}m
        =n^{-1+o(1)}.
    \end{align*}
    The target parameter is
    \(\gamma^\star_{m,s_T}=e^{-c_n^\star(s)}=n^{o(1)}\), so
    \[
        |\bar\gamma_{m,u}-\gamma^\star_{m,s_T}|
        =e^{-c_n^\star(s)}|e^{-\Delta}-1|
        =n^{-1+o(1)}.
    \]
    The coupling proof of Lemma~\ref{lem:comparison-poisson-general}, now
    with common truncation at \(m-1\), yields
    \[
        d_{\mathrm{TV}}(\bar\nu^\star_{m,u},\nu^{\star,T})
        \le n^{-1+o(1)}.
    \]
    Combining the three steps proves the lemma.
\end{proof}

\subsection{Approximate sufficient statistic}\label{subsection:approximate-sufficient-statistic}
We now combine the previous ingredients to prove that the untouched set is an approximate sufficient statistic. This is the star-transposition analogue of \cite[Proposition 2.1]{jain2024hitting}. As in Jain--Sawhney, the proof splits according to whether a permutation has many additional fixed points or not.

\begin{prop} \label{prop:main-prop-bound}
    Let \(s\in\mathbb Z_{\ge0}\), let \(s_n=\lfloor n\log n\rfloor\), and suppose \(s=s_n+nc_n^\star(s)\), where
    \(|c_n^\star(s)|\le \tfrac12\log\log n-4\log\log\log n\).
    Let \(T_n\subseteq[n]_\star\) satisfy \(|T_n|\le(\log n)^2\), and let \(Y_s^{\,T_n}\) denote the random variable \(Y_s\) conditioned on \(\mathcal F(T_n)\).
    Let \(\mu^{T_n}\) be the probability measure on \(\mathfrak S_n\) obtained by embedding the uniform distribution on \(\mathfrak S_{[n]\setminus T_n}\) into \(\mathfrak S_n\) via the natural inclusion.
    Then
    \[ d_{\mathrm{TV}}\!\left(\mathcal L(Y_s^{\,T_n}),\mu^{T_n}\right) \le n^{-1/2+o(1)}, \]
    uniformly over all such \(T_n\).
\end{prop}

\begin{proof}
    The star-transposition walk is invariant under relabelings of the non-hub elements that fix \(1\). We may therefore assume that \(T_n=[M_n]_\star\), where \(M_n=|T_n|+1\). As before, suppress the dependence on \(n\) and write \(M=M_n\), \(Y_s^{M^\star}=Y_s^{[M]_\star}\), \(\mu^{M^\star}=\mu^{[M]_\star}\), and \(\nu^{\star,M}=\nu^{\star,[M]_\star}\). Let \(\mathfrak S_n^{M^\star}\) denote the set of permutations that fix \([M]_{\star}\). 
    By the definition of total variation distance 
    \[d_{\mathrm{TV}} \! \left(\mathcal L(Y_s^{\,M^\star}),\mu^{M^\star}\right) 
    = \frac{1}{2} \sum_{\sigma \in \mathfrak{S}_n^{M^\star}} |\bbP(Y_s^{M^\star} = \sigma) - \mu^{M^\star}(\sigma)| 
    = \frac{1}{2} \sum_{\sigma \in \mathfrak{S}_n^{M^\star}} \left|\frac{\bbP(\{Y_s = \sigma\}\cap\mathcal{F}(M))}{\bbP(\mathcal{F}(M))} - \mu^{M^\star}(\sigma) \right|.\]
    Multiplying both sides by \(\bbP(\mathcal{F}(M))\), it suffices to show 
    \begin{equation}\label{eq:rewrite-eq-prop2.1}
    \sum_{\sigma \in \mathfrak{S}_n^{M^\star}} \left|\bbP(\{Y_s = \sigma\}\cap\mathcal{F}(M)) - \bbP(\mathcal{F}(M)) \cdot \mu^{M^\star}(\sigma) \right| \leq n^{-1/2 + o(1)}\bbP(\mathcal{F}(M)).
    \end{equation}

    We decompose according to the number of fixed non-hub elements beyond
    the forced set \([M]_\star\):
    \[
        \mathcal S_{\mathrm{fix}}:=\left\{\sigma\in\mathfrak S_n^{M^\star}:
        |\mathrm{Fix}_\star(\sigma)\setminus[M]_\star|\le K_n\right\},
        \qquad
        \mathcal L_{\mathrm{fix}}:=\mathfrak S_n^{M^\star}\setminus\mathcal S_{\mathrm{fix}}.
    \]
    Recall that for any \(T \subseteq [n]_\star\),
    \(\widetilde{Y}_s^T\) denotes the random variable \(Y_s\) conditioned
    on \(\mathcal{G}(T)\). In particular, we write
    \(\widetilde{Y}_s^{M^\star} := \widetilde{Y}_s^{[M]_\star}\).

    \smallskip

    \noindent\textbf{Contribution from \(\mathcal L_{\mathrm{fix}}\).}
    By Lemma~\ref{lem:self-reducibility},
    \[ d_{\mathrm{TV}}\!\left(\mathcal L(\widetilde{Y}_s^{\,M^\star}),\nu^{\star,M}\right) \le n^{-1/2+o(1)}.\]
    Restricting the sum to \(\mathcal L_{\mathrm{fix}} \subset \mathfrak S_n^{M^\star}\), using \(|a-b|\ge a-b\), and rearranging the terms, we obtain
    \[ \sum_{\sigma\in\mathcal L_{\mathrm{fix}}} \mathbb P(\widetilde{Y}_s^{\,M^\star}=\sigma) \le n^{-1/2+o(1)} + \nu^{\star,M}(\mathcal L_{\mathrm{fix}}).\]

    It remains to bound \(\nu^{\star,M}(\mathcal L_{\mathrm{fix}})\). Under
    \(\nu^{\star,M}\), let \(Q\) be the number of additional forced fixed
    points, and let \(F\) be the number of fixed points of the uniform
    permutation subsequently sampled on the remaining ground set. Then
    \[
        |\mathrm{Fix}_\star(\sigma)\setminus[M]_\star|\le Q+F.
    \]
    If \(F_r\) is the number of fixed points of a uniform permutation of
    \(r\) elements, then
    \(\mathbb E[(F_r)_j]\le1\) for every \(j\), and therefore, for
    \(\theta\ge0\),
    \[
        \mathbb E e^{\theta F_r}
        =\sum_{j\ge0}\frac{(e^\theta-1)^j}{j!}
        \mathbb E[(F_r)_j]
        \le \exp(e^\theta-1).
    \]
    The truncation probability for
    \(Q\sim\mathrm{Pois}(\gamma_{n,s}^\star)\) is
    \(1-n^{-\omega(1)}\). It follows, uniformly in the size of the
    remaining ground set, that
    \[
        \mathbb E_{\nu^{\star,M}}e^{\theta(Q+F)}
        \le 2\exp\!\left((\gamma_{n,s}^\star+1)(e^\theta-1)\right).
    \]
    Taking \(e^\theta=K_n/(\gamma_{n,s}^\star+1)\), which exceeds \(1\)
    for all sufficiently large \(n\), gives
    \begin{align*}
        \nu^{\star,M}(\mathcal L_{\mathrm{fix}})
        &\le 2\left(\frac{e(\gamma_{n,s}^\star+1)}{K_n}\right)^{K_n}
        \le n^{-1/2+o(1)}.
    \end{align*}
    Here we used
    \(\gamma_{n,s}^\star\le
    \sqrt{\log n}/(\log\log n)^4\). Consequently,
    \begin{equation}\label{eq:final-tilde-X-bound}
        \sum_{\sigma\in\mathcal L_{\mathrm{fix}}}
        \mathbb P(\widetilde Y_s^{\,M^\star}=\sigma)
        \le n^{-1/2+o(1)}.
    \end{equation}

    We now bound the portion of the sum in \eqref{eq:rewrite-eq-prop2.1} indexed by \(\mathcal L_{\mathrm{fix}}\). Since \(\mathcal F(M)\subseteq \mathcal G(M)\),
    \[\sum_{\sigma\in\mathcal L_{\mathrm{fix}}} \mathbb P(\{Y_s=\sigma\}\cap\mathcal F(M)) 
    \le \sum_{\sigma\in\mathcal L_{\mathrm{fix}}} \mathbb P(\{Y_s=\sigma\}\cap\mathcal G(M)).\]
    Conditioning on \(\mathcal G(M)\) and using \eqref{eq:final-tilde-X-bound},
    \[ \sum_{\sigma\in\mathcal L_{\mathrm{fix}}} \mathbb P(\{Y_s=\sigma\}\cap\mathcal G(M))
    \le \mathbb P(\mathcal G(M))\,n^{-1/2+o(1)}
    \le \mathbb P(\mathcal F(M))\,n^{-1/2+o(1)}, \]
    where the last inequality uses \(\mathbb P(\mathcal G(M))\le n^{o(1)}\mathbb P(\mathcal F(M))\), which follows from Lemma~\ref{lem:relationship-f(T)-g(T)}(2) and \(e^{\gamma_{n,s}^\star}=n^{o(1)}\).
    The same exponential-moment estimate with \(Q=0\) gives
    \[
        \mu^{M^\star}(\mathcal L_{\mathrm{fix}})
        \le \left(\frac e{K_n}\right)^{K_n}
        =n^{-1+o(1)}.
    \]
    Therefore
    \[
        \sum_{\sigma \in \mathcal L_{\mathrm{fix}}} \bbP(\mathcal{F}(M))
        \mu^{M^\star}(\sigma)
        \le n^{-1+o(1)}\bbP(\mathcal F(M)).
    \]
    Combining these estimates shows that the total contribution from \(\mathcal L_{\mathrm{fix}}\) is \(n^{-1/2+o(1)}\bbP(\mathcal{F}(M))\).

    \smallskip

    \noindent\textbf{Contribution from \(\mathcal S_{\mathrm{fix}}\).}
    By inclusion--exclusion (Equation~\ref{eq:g-and-f-relationship-indicator}), for each \(\sigma\) we have
    \[\bbP(\{Y_s = \sigma\} \cap \mathcal{F}(M)) = \sum_{[M]_\star \subseteq T \subseteq \mathrm{Fix}_\star(\sigma)} (-1)^{|T| - M + 1} \bbP(\{Y_s = \sigma\} \cap \mathcal{G}(T)).\]
    Since \(\mu^{M^\star}(\sigma)=1/(n-M+1)!\), Lemma~\ref{lem:relate-f-to-g-dot-nu} gives
    \[\bbP(\mathcal{F}(M)) \cdot \mu^{M^\star}(\sigma) = \sum_{[M]_\star \subseteq T \subseteq \mathrm{Fix}_{\star}(\sigma)} (-1)^{|T|-M+1} \, \mathbb{P}(\mathcal{G}(T)) \, \nu^{\star, T}(\sigma)
    +O\!\left(\frac{n^{-1/2+o(1)}\bbP(\mathcal{F}(M))}{(n-M+1)!}\right).\]
    Subtracting these expressions and applying the triangle inequality yields
    \begin{align*}
    \sum_{\sigma \in \mathcal S_{\mathrm{fix}}} &\left|\bbP(\{Y_s = \sigma\}\cap\mathcal{F}(M)) - \bbP(\mathcal{F}(M)) \cdot \mu^{M^\star}(\sigma) \right| \\
    &\leq n^{-1/2+o(1)}\bbP(\mathcal{F}(M)) + \sum_{\sigma \in \mathcal S_{\mathrm{fix}}}\sum_{[M]_\star \subseteq T \subseteq \mathrm{Fix}_{\star}(\sigma)} \left|\bbP(\{Y_s = \sigma\} \cap \mathcal{G}(T)) - \mathbb{P}(\mathcal{G}(T)) \, \nu^{\star, T}(\sigma)\right|.
    \end{align*}
    Switching the order of summation and applying Lemma~\ref{lem:self-reducibility} gives
    \[\leq n^{-1/2+o(1)}\bbP(\mathcal{F}(M)) + 2n^{-1/2 + o(1)}\sum_{[M]_\star \subseteq T,\, |T| \leq M-1+K_n} \bbP(\mathcal{G}(T)). \]
    Here the factor \(2\) comes from the identity \(\sum_\sigma|\mu(\sigma)-\nu(\sigma)|=2d_{\mathrm{TV}}(\mu,\nu)\).
    Finally, Lemma~\ref{lem:relationship-f(T)-g(T)}(3) implies that the last sum is
    \[\leq  n^{-1/2 + o(1)}\bbP(\mathcal{F}(M)).\]
    Combining the \(\mathcal S_{\mathrm{fix}}\) and \(\mathcal L_{\mathrm{fix}}\) contributions proves
    \eqref{eq:rewrite-eq-prop2.1}, and hence the proposition.
\end{proof}

\section{Hitting-time mixing} \label{sec:hitting-time-mixing}

With the deterministic-time approximation from
Section~\ref{sec:deterministic-time-approx-for-star} and the approximate
sufficient statistic from
Section~\ref{sec:approximate-sufficient-statistic} in hand, we now prove the
hitting-time mixing theorem. We stop the walk at a deterministic time \(t^*\)
chosen so that the untouched set \(R_{t^*}\) is typically very small and hence
Proposition~\ref{prop:main-prop-bound} applies. Starting from a uniform
permutation on the touched set, we then run a marking process in the spirit of
a strong uniform time {\cite[Chapter 5]{MR2695907}}.
For the sequence of star transpositions \((1, I_1), (1, I_2), (1, I_3), \ldots\), let \(\tau\) be the random stopping time 
\[\tau := \min \{t : \{2, 3, \ldots ,n\} \subseteq \{I_1, \ldots, I_t\}\}.\]
The proof compares the original star-transposition walk at time \(\tau\) to this auxiliary marking process.

\begin{proof}[ of Theorem~\ref{thm:hitting-time-mixing-star}]
    Set
    \[
        a_n:=\frac12\log\log n-4\log\log\log n,
        \qquad
        t^*:=\left\lceil s_n-na_n\right\rceil.
    \]
    Then \(t^*\) is an integer and
    \(c_n^\star(t^*)=(t^*-s_n)/n\in[-a_n,-a_n+1/n]\), so the hypotheses of Proposition~\ref{prop:main-prop-bound} hold at time \(t^*\). For \(t\ge0\), let
    \[
        R_t:=[n]_\star\setminus\{I_1,\ldots,I_t\}
    \]
    be the set of indices not selected by time \(t\).

    We first define an augmented marking process. Given a permutation-valued chain \((\sigma_t)_{t\ge t^*}\), mark at time \(t^*\) every card in \([n]\setminus R_{t^*}\), and let \(V_t\) denote the set of marked card labels at time \(t\). For \(t\ge t^*\), draw \(I_{t+1}\) uniformly from \([n]\). Before applying the next transposition, if \(\sigma_t(1)\notin V_t\) and either \(I_{t+1}=1\) or \(\sigma_t(I_{t+1})\in V_t\), set
    \(V_{t+1}=V_t\cup\{\sigma_t(1)\}\); otherwise set \(V_{t+1}=V_t\). Then update
    \[
        \sigma_{t+1}=\sigma_t(1,I_{t+1}).
    \]
    Conditional on \(R_{t^*}=T\), let \(Z_{t^*}\sim\mu^T\), and evolve
    \(Z\) using the star-transposition updates and marking rule above. On
    the exceptional event \(\{\tau\le t^*\}\), define \(Z_\tau\)
    arbitrarily.

    We first establish a conditional uniformity property for \(Z\). Let
    \[
        U_t:=\{i\in[n]:Z_t(i)\in V_t\}
    \]
    be the set of positions occupied by marked cards. We claim that,
    conditional on the marking history, the untouched-set history through
    time \(t\), and the complete placement of the unmarked cards, the map
    \(Z_t|_{U_t}:U_t\to V_t\) is a uniformly random bijection. At time
    \(t^*\), this follows from the definition of \(\mu^{R_{t^*}}\).

    We will also use the following invariant. If \(j\in R_t\), then position
    \(j\) has never participated in a transposition. At time \(t^*\), the law
    \(\mu^{R_{t^*}}\) fixes \(j\), and as long as \(j\) remains untouched we
    have \(Z_t(j)=j\); this card is unmarked. Consequently,
    \[
        U_t\cap R_t=\varnothing
    \]
    for every \(t\ge t^*\).

    Suppose the claim holds at time \(t\). If no card is marked at the next step, then either two unmarked cards are transposed, in which case the marked bijection is unchanged, or the top card is marked already. In the latter case, the update either permutes two positions of \(U_t\), or replaces position \(1\) in \(U_t\) by the selected unmarked position. Conditional on the updated unmarked-card placement, each operation is a bijection on the set of assignments of \(V_t\) to \(U_t\), so uniformity is preserved.

    If a new card \(v=Z_t(1)\) is marked, then \(1\notin U_t\) and
    \(I_{t+1}\in\{1\}\cup U_t\). Conditional on the new marking and
    untouched-set histories and on the placement of the remaining unmarked
    cards, these \(|V_t|+1\) choices are equally likely: by the preceding
    invariant, none of the positions in \(\{1\}\cup U_t\) is untouched, so the untouched-set history does not
    distinguish among them. Moreover, all these choices leave the placement
    of every still-unmarked card other than \(v\) unchanged, so conditioning
    on that placement introduces no additional bias. After the update,
    \(V_{t+1}=V_t\cup\{v\}\) and \(U_{t+1}=U_t\cup\{1\}\); the choice of
    \(I_{t+1}\) determines the position of \(v\) uniformly among the
    \(|V_t|+1\) positions in \(U_{t+1}\). Together with the inductive
    hypothesis, this makes \(Z_{t+1}|_{U_{t+1}}\) a uniform bijection. This
    proves the claim.

    \smallskip
    \noindent\emph{The law of \(Z_\tau\).}
    For \(k\in[n]_\star\), let \(J_k\) indicate that \(k\) is not selected among \(I_1,\ldots,I_{t^*}\). Then
    \[
        \Lambda_n:=\mathbb E|R_{t^*}|
        =(n-1)\left(1-\frac1n\right)^{t^*}
        =(1+o(1))(\log n)^{1/2}(\log\log n)^{-4}.
    \]
    View the choices \(I_1,\ldots,I_{t^*}\) as balls thrown independently into
    \(n\) bins, and let \(B_k\) be the occupancy of bin \(k\). The occupancy
    vector is negatively associated by \cite[Theorem 13]{dubhashi1996balls},
    and \cite[Proposition 7(2)]{dubhashi1996balls} shows that the same is true
    of the indicators \(1-J_k=\mathbf 1_{\{B_k\ge1\}}\). The marginal product
    bound \cite[Proposition 4]{dubhashi1996balls} therefore gives
    \[
        \mathbb P(\tau\le t^*)
        =\mathbb P(R_{t^*}=\varnothing)
        \le\prod_{k=2}^n\mathbb P(J_k=0)
        \le e^{-\Lambda_n}
        =\exp\bigl(-(\log n)^{1/2+o(1)}\bigr).
    \]

    Put \(\ell_n:=\lfloor(\log n)^2\rfloor\) and \(H_n:=\lceil2n\log n\rceil\), and define
    \[
        E_1:=\{|R_{t^*}|\le\ell_n\},
        \qquad
        E_2:=\{\tau-t^*\le H_n\}.
    \]
    For every integer \(m\ge1\), a union bound gives
    \[
        \mathbb P(|R_t|\ge m)
        \le\binom{n-1}{m}\left(1-\frac mn\right)^t.
    \]
    Applying this with \(m=\ell_n\) and \(t=t^*\) yields
    \(\mathbb P(E_1^c)=n^{-\omega(1)}\). Moreover,
    \[
        \mathbb P(E_2^c)
        \le\mathbb P(|R_{t^*+H_n}|>0)
        \le(n-1)\left(1-\frac1n\right)^{t^*+H_n}
        \le n^{-2+o(1)}.
    \]

    Let \(B\) be the event that, during the first \(H_n\) steps after \(t^*\), two consecutive updates select previously untouched indices: for some \(r\in\{t^*+1,\ldots,t^*+H_n-1\}\),
    \(I_r\in R_{r-1}\) and \(I_{r+1}\in R_r\). On \(E_1\), the untouched set has size at most \(\ell_n\) throughout this interval. Conditioning on the history through time \(r-1\) and then using the independence of \(I_r\) and \(I_{r+1}\) gives
    \[
        \mathbb P(B\cap E_1)
        \le H_n\left(\frac{\ell_n}{n}\right)^2
        =n^{-1+o(1)}.
    \]
    Define
    \[
        E:=\{\tau>t^*\}\cap E_1\cap E_2\cap B^c.
    \]
    We justify the marking assertion by induction over first-touch times. At
    time \(t^*\), the unmarked cards are precisely the labels in \(R_{t^*}\),
    each in its own untouched position. If \(I_r\in R_{r-1}\), then the update
    at time \(r\) moves the unmarked card \(I_r\) to position \(1\), and it is
    the only unmarked card outside \(R_r\). If \(r<\tau\), then on \(B^c\) we
    have \(I_{r+1}\notin R_r\). Hence either \(I_{r+1}=1\) or position
    \(I_{r+1}\) contains a marked card, so the marking rule marks \(I_r\)
    before the next transposition. Since \(R_{r+1}=R_r\), this restores the
    property that the unmarked cards are precisely the labels in \(R_{r+1}\),
    each in its own untouched position. Consequently, on \(E\), every newly
    touched card other than the last is marked at the following update. At time
    \(\tau\), the only unmarked card is \(L:=I_\tau\), and it occupies
    position \(1\). The stopping time \(\tau\), the event \(E\), and the
    identity of \(L\) are all determined by the untouched-set history. The
    fixed-time conditional-uniformity claim therefore applies on each event
    \(\{\tau=r\}\); summing over \(r\) shows that, given \(E\) and \(L\), the
    other \(n-1\) cards form a uniform bijection onto the other \(n-1\)
    positions.
    Moreover, \(E\) is invariant under every
    relabeling of \([n]_\star\), so \(L\), conditional on \(E\), is uniform
    on \([n]_\star\). It follows that
    \[
        \mathcal L(Z_\tau\mid E)
        =U_{\mathfrak S_n}\bigl(\,\cdot\mid \sigma(1)\ne1\bigr).
    \]
    The preceding bounds give
    \[
        \mathbb P(E^c)
        \le\mathbb P(\tau\le t^*)+\mathbb P(E_1^c)
        +\mathbb P(E_2^c)+\mathbb P(B\cap E_1)
        \le\exp\bigl(-(\log n)^{1/2+o(1)}\bigr).
    \]
    Since the uniform measure assigns mass \(1/n\) to
    \(\{\sigma:\sigma(1)=1\}\), we conclude that
    \begin{equation}\label{eq:auxiliary-hitting-bound}
        d_{\mathrm{TV}}(\mathcal L(Z_\tau),U_{\mathfrak S_n})
        \le \mathbb P(E^c)+\frac1n
        \le\exp\bigl(-(\log n)^{1/2+o(1)}\bigr).
    \end{equation}

    \smallskip
    \noindent\emph{Comparison of \(Y_\tau\) and \(Z_\tau\).}
    Write \(p_T:=\mathbb P(R_{t^*}=T)\). Conditional on \(R_{t^*}=T\),
    \[
        \mathcal L(Y_{t^*}\mid R_{t^*}=T)=Y_{t^*}^{\,T},
        \qquad
        \mathcal L(Z_{t^*}\mid R_{t^*}=T)=\mu^T.
    \]
    Lemma~\ref{lem:TV-of-mixtures}, Proposition~\ref{prop:main-prop-bound}, and the bound on \(E_1^c\) give
    \begin{align*}
        d_{\mathrm{TV}}(\mathcal L(Y_{t^*}),\mathcal L(Z_{t^*}))
        &\le\sum_{T\subseteq[n]_\star}p_T
        d_{\mathrm{TV}}(\mathcal L(Y_{t^*}^{\,T}),\mu^T)\\
        &\le\mathbb P(E_1^c)
        +\sum_{|T|\le\ell_n}p_T n^{-1/2+o(1)}
        \le n^{-1/2+o(1)}.
    \end{align*}

    More precisely, sample a common value \(R_{t^*}=T\) and, conditional
    on this value, take a maximal coupling \((W,V)\) of
    \(Y_{t^*}^{\,T}\) and \(\mu^T\). The preceding estimate gives
    \(\mathbb P(W\ne V)\le n^{-1/2+o(1)}\). Use the same subsequent
    choices \(I_{t^*+1},I_{t^*+2},\ldots\) for both chains. On
    \(\{\tau>t^*\}\cap\{W=V\}\), the two states agree at time \(\tau\).
    Therefore,
    \begin{equation}\label{eq:coupling-hitting-bound}
        d_{\mathrm{TV}}(\mathcal L(Y_\tau),\mathcal L(Z_\tau))
        \le\mathbb P(\tau\le t^*)+\mathbb P(W\ne V)
        \le n^{-1/2+o(1)}.
    \end{equation}

    By the triangle inequality for total variation and
    \eqref{eq:coupling-hitting-bound}, \eqref{eq:auxiliary-hitting-bound},
    which follow from Lemma~\ref{lem:TV-of-mixtures},
    Proposition~\ref{prop:main-prop-bound}, and the maximal coupling
    inequality, we have
    \begin{equation}
        \begin{aligned}
        d_{\mathrm{TV}}(\mathcal L(Y_{\tau}),U_{\mathfrak S_n})
        &\le d_{\mathrm{TV}}(\mathcal L(Y_\tau),\mathcal L(Z_\tau))
          +d_{\mathrm{TV}}(\mathcal L(Z_\tau),U_{\mathfrak S_n})\\
        &\le n^{-1/2+o(1)}
          +\exp\bigl(-(\log n)^{1/2+o(1)}\bigr)\\
        &\le \exp\bigl(-(\log n)^{1/2+o(1)}\bigr).
        \end{aligned}
    \end{equation}
    as claimed.
\end{proof}

\bibliographystyle{plain}
\bibliography{citations}
\end{document}